\documentclass[reqno, oneside, a4paper, 11pt]{amsart}
\usepackage{etoolbox}
\patchcmd{\section}{\normalfont}{\normalfont\large}{}{}
\makeatletter
\patchcmd{\@setauthors}{\footnotesize}{\footnotesize}{}{}
\makeatother
\makeatletter
\patchcmd{\@settitle}{\bfseries}{\bfseries\large}{}{}
\makeatother

\usepackage[foot]{amsaddr}

\usepackage{amsmath,amssymb,amsthm, mathdots}
\usepackage{mathtools}
\usepackage{amsfonts}
\usepackage{physics}

\usepackage[T1]{fontenc}
\usepackage[utf8]{inputenc}
\usepackage[english]{babel}

\usepackage{kpfonts}
\usepackage[spacing, kerning]{microtype}
\microtypecontext{spacing=nonfrench}
\usepackage{bm}
\usepackage[autostyle]{csquotes}
\usepackage{tikz}
\usepackage{tikzscale}

\usepackage{enumitem}
\usepackage[abbrev, alphabetic, nobysame]{amsrefs}
\usepackage[bookmarksnumbered,linktocpage,hypertexnames=false,colorlinks=true,linkcolor=blue,urlcolor=blue,citecolor=purple,anchorcolor=green,breaklinks=true]{hyperref}
\usepackage{tabularx}
\usepackage{cleveref}
\Crefname{figure}{Figure}{Figures}
\usepackage{xcolor}
\usepackage{multicol}
\usepackage{multirow}
\usepackage{centernot}

\newtheorem*{theorem*}{Theorem}
\newtheorem{theoremintro}{Theorem}

\newtheorem{theorem}{Theorem}[subsection]
\newtheorem{lemma}[theorem]{Lemma}
\newtheorem{proposition}[theorem]{Proposition}
\newtheorem{corollary}[theorem]{Corollary}

\theoremstyle{definition}
\newtheorem{definition}[theorem]{Definition}
\newtheorem{example}[theorem]{Example}

\newtheorem{exampleintro}[theoremintro]{Example}

\crefname{theoremintro}{Theorem}{Theorems}

\DeclarePairedDelimiter{\ceil}{\lceil}{\rceil}

\newcommand{\N}{\mathrm N}
\newcommand{\Z}{\mathbb Z}
\newcommand{\C}{\mathbb C}
\newcommand{\PP}{\mathbb P}
\newcommand{\Q}{\mathbb Q}
\newcommand{\R}{R}

\def\CC{\mathbb{C}}
\def\NN{\mathbb{N}}
\def\RR{\mathbb{R}}
\def\ZZ{\mathbb{Z}}

\def\SS{\mathbb{S}}

\renewcommand{\epsilon}{\varepsilon}
\renewcommand{\emptyset}{\varnothing}

\newcommand{\e}{\mathrm e}
\newcommand{\cH}{\mathcal{H}}

\newcommand{\cM}{\mathcal{M}}
\newcommand{\scrM}{\mathscr{M}}

\DeclareMathOperator{\vol}{\mathrm{vol}}

\DeclareMathOperator{\Pic}{Pic}

\DeclareMathOperator{\conv}{conv}

\DeclareMathOperator{\jac}{J}
\DeclareMathOperator{\supp}{supp}

\DeclareMathOperator{\len}{len}

\setlist[enumerate,1]{label={(\roman*)},ref={\thetheorem (\roman*)}}

\newcommand{\inner}[2]{\langle {{#1},{#2}} \rangle}

\DeclarePairedDelimiterX{\BW}[2]{\langle}{\rangle_{BW}}{#1, #2}
\usepackage{booktabs} 
\usepackage{cleveref}
\usepackage{comment,soul}
\usepackage[
  margin=3cm,
  includefoot,
  footskip=30pt,
]{geometry}
\title{Totally real divisors on curves}
\author{Lorenzo Baldi\textsuperscript{*} \and Mario Kummer\textsuperscript{$\dagger$} \and Daniel Plaumann\textsuperscript{$\ddagger$}}
\address{\textsuperscript{*}Universität Leipzig and MPI MiS, Leipzig, Germany, lorenzo.baldi@mis.mpg.de}
\address{\textsuperscript{$\dagger$}TU Dresden, Germany, mario.kummer@tu-dresden.de}
\address{\textsuperscript{$\ddagger$}TU Dortmund, Germany, daniel.plaumann@math.tu-dortmund.de}
\date{September 2025}
\begin{document}
\begin{abstract}
    Since the works of Krasnov and Scheiderer, there has been an interest in studying effective totally real divisors on a curve $X$ defined over a real closed field, i.e.,~effective divisors supported on the real locus. Scheiderer proved that, for smooth curves over $\RR$ with nonempty real locus, each divisor of sufficiently high degree is linearly equivalent to an effective totally real one. The smallest degree $\mathrm{N}(X)$ with this property is called the \emph{totally real divisor threshold}.

    When the field is non-Archimedean, we obtain a classification of topological types of smooth curves for which $\N(X)$ can be ${\infty}$.
    As a consequence, for curves over $\RR$ we prove that $\N(X)$ cannot be bounded from above only in terms of the topological type, unless $X(\RR)$ has many connected components.
    We complement this qualitative result with a quantitative lower bound for $\N(X)$, depending on metric properties of the Jacobian and the curve in the Bergman metric. Finally, we relate these metric properties to period matrices of $X$, expressed in a way compatible with the real structure.
\end{abstract}
\maketitle

\section{Introduction}
Let $X$ be a projective algebraic curve defined over $\RR$. If the real locus $X(\RR)$ is nonempty, we ask if a given (conjugation invariant) divisor on $X$ is linearly equivalent to a divisor supported on $X(\RR)$. It follows from the work of Krasnov \cite{krasnovAlbaneseMapGMZvarieties1984}*{Sec.~2} that the answer is always positive, provided that the curve is smooth and the degree of the divisor is large enough.  
Scheiderer refined this result, showing that such a divisor can be chosen to be effective.
\begin{theorem*}[{\cite{scheidererSumsSquaresRegular2000}*{p.~1050}}]
  Let $X$ be a smooth irreducible curve over $\RR$ such that $X(\RR) \neq \emptyset$. Then every divisor of sufficiently high degree is linearly equivalent to an effective divisor supported on $X(\RR)$.
\end{theorem*}

This result motivates the following definition, which can be stated for arbitrary real closed fields. 

\begin{definition}
    \label{def:real_bound}
    Let $R$ be a real closed field and let $X$ be a curve over $R$. We say that a divisor $D$ on $X$ is \emph{totally real} if $\supp(D)\subset X(R)$. We define the \emph{totally real divisor threshold} of $X$ as the infimum $m \in \NN$ such that every divisor of degree $m$ is linearly equivalent to an effective totally real one, and we denote it by $\N(X)$.
\end{definition}

Notice that, by definition, $\N(X) = {\infty}$ if $X(R) = \emptyset$. Scheiderer's result states that $\N(X)<{\infty}$ for smooth curves over $\RR$ with $X(\RR)\neq \emptyset$. 
On the other hand, for singular curves $\N(X)$ need not be finite (see \cite{monnierRealGeneralizedJacobian2005}). It is a natural question what data the invariant $\N(X)$ depends on in general.

In the article of Scheiderer, it was also incorrectly claimed that the finiteness of $\N(X)$ in the smooth case extends to arbitrary real closed fields $R$ \cite{scheidererSumsSquaresRegular2000}*{Th.~2.7}. The first counterexample to this statement, i.e., a smooth curve over a real closed field $R$ with $X(R)\neq \emptyset$ and $\N(X) = {\infty}$, was  recently found by Benoist and Wittenberg \cite{benoistIntegralHodgeConjecture2020}*{Rem.~9.26}.
Our first contribution is to completely characterize when such counterexamples exist.

\begin{theoremintro}[{see \Cref{thm:A_complete} for a refined statement}]
    \label{thm:A}
Let $R$ be a real closed field.
    \begin{enumerate}
        \item If $X$ is a smooth curve over $R$ of genus $g$ with $\N(X)=\infty$ such that $X(R)$ has $r>0$ (semialgebraically) connected components, then $r<g$ and $R$ is non-Archimedean.
        \item Conversely, if $R$ is non-Archimedean and $0\leq r<g$, then there exists a smooth curve $X$ over $R$ of genus $g$ with $\N(X)=\infty$ such that $X(R)$ has $r$ (semialgebraically) connected components.
    \end{enumerate}
\end{theoremintro}

Our next result is better stated using real moduli spaces. The coarse moduli space $\scrM_g^{\RR}$ of real isomorphism classes of smooth real algebraic curves was introduced by Gross and Harris \cite{grossRealAlgebraicCurves1981a} and Sepp\"al\"a and Silhol \cites{seppalaModuliSpacesReal1989}. It is a semianalytic variety \cite{huismanRealQuotientSingularities1999} of dimension $3g-3$, which is homeomorphic with the underlying topological space of the corresponding real algebraic stack \cite{degaayfortmanRealModuliSpaces2022a}*{Th.~8.2}.

The \emph{topological type} of a curve $X$ (see \Cref{sec:preliminaries} for more details) is the triple $(g,r,a)$, where $g$ is the genus of $X$, $r$ is the number of connected components of $X(\RR)$ and $a = 1$ if $X \setminus X(\RR)$ is connected and $a=0$ if $X\setminus X(\RR)$ has two connected components. 
If two curves are isomorphic over $\RR$, then they have the same topological type. Moreover, two isomorphism classes of smooth curves lie in the same connected component of $\scrM_g^\RR$ if and only if they have the same topological type.
We can write the moduli space as the disjoint union $\scrM_g^\RR = \bigsqcup_{(g,r,a)} \scrM_{(g,r,a)}^\RR$, where $\scrM_{(g,r,a)}^\RR$ parametrizes real isomorphism classes of smooth curves with the topological type $(g,r,a)$. We say that a topological type $(g,r,a)$ is \emph{admissible} if $\scrM_{(g,r,a)}^\RR$ is nonempty.

\begin{theoremintro}[{see \Cref{sec:qualitative}}]
    \label{thm:B}
    Consider the totally real divisor threshold $\N$ as a function on $\scrM_g^\RR$.
    \begin{enumerate}
        \item For $m \ge 2g+1$, the set 
        $\N^{-1}(\{ \, 0, \dots , m \,\}) = \left\{ \, X \in \scrM_g^{\RR} \mid \N(X) \le m  \, \right\}$
        is closed. 
        \item For an admissible topological type $(g,r,a)$ with $r>0$ real connected components, let \[\N_{(g,r,a)} \colon \scrM_{(g,r,a)}^\RR \longrightarrow \NN\] be the restriction of $\N$ to $\scrM_{(g,r,a)}^\RR$. Then $\N_{(g,r,a)}$ is unbounded unless $r=g$ or $r=g+1$.
\end{enumerate}
\end{theoremintro}

The two bounded cases $r=g$ and $r=g+1$ were already found by Huisman \cite{huismanGeometryAlgebraicCurves2001} and Monnier \cite{monnierDivisorsRealCurves2003}, who also introduced the notation $\N(X)$. In both cases, they were able to show the upper bound $\N(X) \le 2g-1$. We reprove this upper bound in a unified way in \Cref{thm:huismanmonnier}.

The interest in bounding $\N(X)$ also arises from applications. For instance, given an embedded curve $X\subset \PP^n$, $\N(X)$ is used to determine the maximal number of real intersection points of $X$ with a hypersurface, or to study the cone of nonnegative forms on $X$ and the dual moment problem. We refer the reader to \cites{leComputingTotallyReal2022c,didioMultidimensionalTruncatedMoment2021b,baldiNonnegativePolynomialsMoment2024b} for more details. The totally real divisor threshold is also employed in the study of hyperbolic embeddings \cite{kummerSeparatingSemigroupReal2020} of curves. The question of which theta characteristics have a totally real representative was studied in the context of convex hulls of canonical curves \cite{tottheta}.

We remark that \Cref{thm:B} (ii) disproves a conjecture of Huisman \cite{huismanNonspecialDivisorsReal2003}*{Conj.~3.4} on so-called unramified curves in $\PP^n$ for odd $n$. Indeed, \cite{monnierDivisorsRealCurves2003}*{Th.~3.7} shows that $\N_{(g,g-1,a)}$ is bounded if the conjecture holds true for $n=g+1$ in case $g$ is even and $n=g+2$ when $g$ is odd. 
Huisman's conjecture was already disproven in the case $n=3$ by Mikhalkin and Orevkov, see \cite{mikhalkinMaximallyWrithedReal2019} and \cite{mikhalkinRigidIsotopyMaximally2021}*{Rem.~2.7}, and  by Kummer and Manevich \cite{kummerHuismansConjecturesUnramified2021}. Our results imply the existence of counterexamples in $\PP^n$ for any odd $n\geq3$. 

\Cref{thm:B} implies that, unless the curve $X$ has many real connected components, $\N(X)$ cannot be bounded just in terms of the invariants $g$, $r$ and $a$ of $X$. However, in \Cref{thm:C} below, we are able to prove that $\N(X)$ can be bounded from below by metric properties of the embedding of the curve in its Jacobian through the Abel--Jacobi map.

Before stating the result, we need some preliminaries: we refer the reader to \Cref{sec:periods} for more details. We denote by $J = \jac(X)\cong \CC^g/\Lambda$ the Jacobian of a curve $X$ of genus $g$. If $X$ is defined over $\RR$ and has $r>0$ real connected components, then $J$ is defined over $\RR$, and $J(\RR)$ is a $g$-dimensional, compact, real Lie group with $2^{r-1}$ connected components. We let $J(\RR)_0$ denote the connected component of the identity. 
The principal polarization of $J$ induces an Hermitian (Kähler) metric on $J$. Its real part is a Riemannian metric on the underlying $2g$-dimensional real manifold, which we call the canonical Riemannian metric. We denote by $\vol(J(\RR)_0)$ the volume of the submanifold $J(\RR)_0 \subset J$ in this metric. The Abel--Jacobi map $\phi \colon X \to J$ is an embedding, hence we can define the \emph{Bergman (Riemannian) metric} on $X$ (see e.g.~\cites{wentworthAsymptoticsArakelovGeensFunction1991,habermannRiemannianMetricsTeichmuller1996a}) as the pullback of the canonical Riemannian metric on $J$ via $\phi$. We denote by $\len(X(\RR))$ the length of the real locus of $X$ with respect to this metric.

We are now ready to state the lower bound for $\N(X)$. A simplified version for $r=1$ is given in \Cref{thm:lower_bound_r=1}.
\begin{theoremintro}[{see \Cref{sec:lower_bound}}]
    \label{thm:C}
    Let $X$ be a smooth curve over $\RR$ of genus $g$ such that $X(\RR)$ has $r >0$ connected components.  Denoting Euler's number by $\e$, we then have:
    \[
        \N(X) \ge 2 \left( 1- r + \left(\left(\frac{r}{2}\right)^{g}\left(\frac{r-1}{\e}\right)^{g(r-1)} \frac{\vol(J(\RR)_0)}{\len(X(\RR))^g} \right)^{\frac{1}{gr}} \right)-1
    \]
\end{theoremintro}
When the topological type of $X$ is fixed, the only ingredient of the bound in \Cref{thm:C} that may vary is the ratio $\frac{\vol(\jac(\RR)_0)}{\len(X(\RR))^g}$. This tells us that the smaller the image of $X(\RR)$ via the Abel--Jacobi map is, compared to the real part of the Jacobian, the larger the totally real divisor threshold will be. To the best of our knowledge, this is the first occurrence of a metric property which is shown to play a key role in the study of divisors on real curves.

It is natural to ask whether \Cref{thm:C} gives an effective way to compute a lower bound for $\N(X)$ and whether or not the bound is tight. As a first step, in \Cref{thm:D} below we show how to read $\vol(\jac(\RR)_0)$, or equivalently $\vol(\jac(\RR))$, from a period matrix for the Jacobian $J = \jac(X) \cong H^0(X, \Omega)^*/ H_1(X, \ZZ) \cong \CC^g/ \Lambda$.
The lattice $\Lambda$ defining $J$ is generated by the columns of the $g\times 2g$ period matrix for $X$. If the bases of $H^0(X, \Omega)^*$ and $H_1(X, \ZZ)$ are chosen in a compatible way with respect to the real structure (see \Cref{sec:periods} for details), then the period matrix takes the special form
\begin{equation*}
    \left( \begin{array}{ccc|ccc}
        1 & {} & {} & {} & {} & {}\\
        {} & {\ddots} & {} & \multicolumn{3}{c}{\frac{1}{2}M + iT} \\
        {} & {} & {1} & {} & {} & {}
    \end{array} \right)
\end{equation*}
with $T \in \RR^g$ a symmetric, positive definite matrix, and $M\in \ZZ^{g\times g}$, called \emph{reflection matrix}. We can now state our result.
\begin{theoremintro}[{see \Cref{sec:Bergman}}]
    \label{thm:D}
    \label{cor:volume_jacobian}
    Let $X$ be a smooth curve over $\RR$ with $X(\RR) \neq \emptyset$, $J = \jac(X)$ and $T$ as above. If $r >0$ is the number of connected components of $X(\RR)$, then
    \[
        \vol(J(\RR)) = 2^{r-1}(\det T)^{-\frac{1}{2}}
    \]
    and $\det T = \vol(J(\RR)_0)^{-2}$ is an invariant of $X$.
\end{theoremintro}
This result is stated for Jacobians for ease of presentation. It holds true more generally for principally polarized abelian varieties over $\RR$, see \Cref{thm:volume_identity} and \Cref{cor:volume}.

Compared to $\vol(J(\RR)_0)$, the computation of $\len(X(\RR))$ (both appearing in \Cref{thm:C}) is more challenging in general. However, in the hyperelliptic case $\len(X(\RR))$ can be computed effectively. In particular, in \Cref{sec:genus2} we compute $\len(X_\epsilon(\RR))$ for a special family of curves $X_\epsilon$ of genus $2$.
\begin{exampleintro}[see \Cref{sec:genus2}]
 Consider the family of curves $X_\epsilon$ with hyperelliptic model \[w^2 = (1+z^2)((1-\epsilon)^2+z^2)((1+\epsilon)^2+z^2)\]
 which is a concrete realization of \cite{benoistIntegralHodgeConjecture2020}*{Rem.~9.26}. Then, for all $0<\epsilon<\frac{1}{2}$, we have 
 \begin{equation*}
    \N(X_\epsilon) \ge  \frac{1}{\sqrt{\epsilon}} ,
\end{equation*}
 see \Cref{sec:example_1}. By a refined analysis, we furthermore prove in \Cref{sec:refined} an upper bound for $\N(X_\epsilon)$ which grows at the same order $1/\sqrt{\epsilon}$. This shows that our general lower bound is asymptotically tight for this example. We will also give a concrete family of divisors on $X_\epsilon$ that realizes the above lower bound on $\N(X_\epsilon)$.
\end{exampleintro}

\subsection{Organization} In \Cref{sec:preliminaries}, we specify the notations and the general assumptions in the manuscript, and give some equivalent definitions for the totally real divisor threshold.

In \Cref{sec:periods}, we discuss Jacobians of real curves and, more generally, principally polarized abelian varieties over $\RR$. In \Cref{sec:comessatti} we give an explicit proof of an old result due to Comessatti on the existence of period matrices in a special form compatible with the real structure. 
In \Cref{sec:volume} we read from such period matrices the volume of the real part of an abelian variety, with respect to the canonical measure induced by the principal polarization. In \Cref{sec:Bergman} we apply the previous results to Jacobians and prove \Cref{thm:D}.

In \Cref{sec:qualitative} we study the totally real divisor threshold $\N$. In \Cref{sec:families} we relate $\N$ for families of curves over $\RR$ and for curves defined over the field of algebraic Puiseux series. In \Cref{sec:few_ovals} we prove \Cref{thm:A,thm:B}.

In \Cref{sec:quantitiative} we give quantitative bounds for $\N(X)$ for curves over $\RR$. In \Cref{sec:many_ovals} we reprove results due to Huisman and Monnier in a unified way, for curves with many real connected components. In \Cref{sec:lower_bound} we prove \Cref{thm:C}.

In \Cref{sec:genus2} we study a special family of curves of genus $2$. In \Cref{sec:example_1} we apply our general results to this family, and finally in \Cref{sec:refined} we prove that our estimate is asymptotically tight for this example.

\subsection{Notations and preliminaries}
\label{sec:preliminaries}
For general references on real algebraic geometry, and in particular on real algebraic curves and their Jacobians, we refer the reader to \cites{grossRealAlgebraicCurves1981a,scheidererRealEtaleCohomology1994,cilibertoRealAbelianVarieties1996a,bochnakRealAlgebraicGeometry1998a,natanzonModuliRealAlgebraic1999, mangolteRealAlgebraicVarieties2020a}.

Let $R$ be a real closed field. We denote by $\PP^n_R$ the $n$-dimensional projective space, seen as a scheme over $R$, and in particular $\PP^n = \PP^n_\RR$. A \emph{curve over} $R$ is a one-dimensional, connected,  projective variety (reduced separated scheme of finite type) over $R$, usually denoted by $X$. Unless otherwise stated, curves over $R$ are not assumed to be smooth or irreducible. A \emph{real curve} is a curve over $\RR$.

The set of $R$-points of a curve $X$ over $R$, also called the \emph{real locus}, is denoted $X(R)$. A (Weil or Cartier) divisor on $X$ is called \emph{totally real} if its support is included in $X(R)$.

If $R=\RR$ and the curve $X$ is smooth, then $X(\RR)$ is either empty or a disjoint union of smooth one-dimensional manifolds, each diffeomorphic to a circle $\SS^1$. Each such curve can equivalently be described as a smooth complex curve, which coincides with the complexification $X_\CC$ of $X$, equipped with an antiholomorphic involution $\sigma \colon X_\CC \to X_\CC$, and $X(\RR) = \mathrm{fix}(\sigma)$. If $X(\RR) \neq \emptyset$, we usually write $X(\RR) = S_1 \sqcup \dots \sqcup S_r$ with $S_i \cong \SS^1$.

\noindent The \emph{topological type} of a smooth curve $X$ over $R$ is the triple $(g,r,a)$, where
\begin{itemize}
    \item $g=g(X)$ is the genus of $X$;
    \item $r=r(X)$ is the number of (semialgebraically) connected components of $X(R)$;
    \item $a=a(X) = \begin{cases}
        0 & \text{ if $X \setminus X(R)$ has two connected components;}\\
        1 & \text{ if $X \setminus X(R)$ is connected.}
        \end{cases}$
\end{itemize}
We will call smooth curves with $a(X)=0$ \emph{dividing}, and those with $a(X)=1$ \emph{non dividing}. The admissible topological types, i.e., those for which there exists a smooth curve over $R$ with topological type $(g,r,a)$, are exactly those which satisfy:
\begin{enumerate}
    \item $0\le r \le g+1$;
    \item if $r=0$ then $a=1$, and if $r=g+1$ then $a=0$;
    \item if $a=0$ then $r \equiv g+1\ \mathrm{mod}\ 2$;
\end{enumerate}
see e.g.~\cites{grossRealAlgebraicCurves1981a} for real curves and \cite{scheidererRealEtaleCohomology1994}*{Sec.~20.1.6} more generally for curves over $R$.

If $X$ is a curve over $\R$, we denote by $\mathrm{Pic}(X)$ the Picard group of $X$. If $X$ is smooth, we let $J = \jac(X) \coloneqq \mathrm{Pic}^0(X)$ denote the Jacobian of $X$. It is an abelian variety over $R$ (see also \Cref{sec:periods} for the case $R = \RR$). Assume $X(R)\neq \emptyset$ and fix $P_0 \in X(R)$. The Abel--Jacobi map is the embedding
\begin{align*}
    \phi \colon X & \longrightarrow \jac(X) \\
    P & \longmapsto [P-P_0]
\end{align*}
and it is defined over $R$. In particular, $\phi(X(R)) \subset \jac(X)(R)$. The qualitative and quantitative relationship between $\phi(X(R))$ and $\jac(X)(R)$ is our main topic.

For any subset $A \subset \jac(X)$ and $m\in\NN$, we write $mA = \{ \, a_1+\cdots +a_m\in \jac(X) \colon a_1,\ldots,a_m \in A \, \}$.

\begin{lemma}
    \label{lem:0}
    Let $X$ be a smooth curve over $R$ such that $X(R) \neq \emptyset$, and let $\phi \colon X \to J=\jac(X)$ be the Abel--Jacobi map with base point $P_0 \in X(R)$. Let:
    \begin{itemize}
        \item $\begin{aligned}[t] N_1 = \inf\big\{ \, m \in \NN \mid \text{ every divisor of degree $m$ is linearly} & \text{ equivalent} \\[-.5em] & \text{to an effective totally real one} \,\big\}
        \end{aligned}$
        \item $\begin{aligned}[t] N_2 = \inf\big\{ \, m \in \NN \mid \text{ every divisor of degree at least $m$ is} & \text{ linearly equivalent} \\[-.5em] & \ \text{to an effective totally real one} \,\big\}
        \end{aligned}$
        \item  $N_3 = \inf\big\{ \, m \in \NN \mid m \phi(X(R)) = J(R)\,\big\}
        $
        \item  $N_4 = \inf\big\{ \, m \in \NN \mid \text{ for all } m'\ge m, \ m' \phi(X(R)) = J(R)\,\big\}
        $
    \end{itemize}
    Then $N_1=N_2=N_3=N_4$.
\end{lemma}
\begin{proof}
    Sending a divisor $D$ to the class $[D-(\deg D)P_0]$, by definition of the Jacobian and the Abel--Jacobi map, shows that $N_1=N_3$ and $N_2=N_4$. Since $P_0 \in X(R)$, we have $0 = [P_0-P_0] \in \phi(X(R))$. This implies $m \phi(X(R))\subset (m+1)\phi(X(R)) \subset J(R)$, proving that $N_3=N_4$.
\end{proof}
The totally real divisor threshold, introduced in \Cref{def:real_bound}, for smooth real curves then coincides with any of the integers in \Cref{lem:0}.

\section*{Acknowledgements}
We thank O.~Benoist for interesting and helpful discussions, in particular about \Cref{ex:div}, {and O.~de Gaay Fortman for his help in deducing \Cref{thm:B} (i) from \Cref{thm:semicontinuity}}. We also thank R.~de~Jong for clarifications and bibliographical references on the Bergman metric. 
L.~Baldi was funded by the Humboldt Research Fellowship for postdoctoral researchers. M.~Kummer was partially supported by DFG grant 502861109.

\section{Periods and metric properties of real curves and abelian varieties}
\label{sec:periods}
Let $X$ be a smooth real curve (i.e., a smooth curve over $\RR$, see \Cref{sec:preliminaries}), and let $\sigma \colon X_\CC \to X_\CC$ be the antiholomorphic involution. The antiholomorphic involution $\sigma$ acts naturally on the homology and cohomology of $X(\CC)$, see \cites{cilibertoRealAbelianVarieties1996a,mangolteRealAlgebraicVarieties2020a}: if $\gamma \in H_1(X(\CC), \ZZ)$ and $\omega \in H^0(X_\CC, \Omega)$, then we denote their images by $\gamma^\sigma$ and $\omega^\sigma$.
By abuse of notation, we will denote $H_1(X(\CC), \ZZ)$ by $H_1(X, \ZZ)$.
The complexification $J_\CC$ of the Jacobian $J = \jac(X)$ is isomorphic to the quotient $J_\CC \cong  H^0(X_\CC, \Omega)^* / H_1(X, \ZZ)$ and inherits an action of $\sigma$.

\subsection{Periods of real curves}
\label{sec:comessatti}
Given a smooth curve $X$ over $\RR$, in this section we discuss how to choose a basis of $H_1(X, \ZZ)$ which is compatible with $\sigma$.
We first notice the following lemma which is a consequence e.g.~of \cite{buserCounterexamplesSurfaceHomology2024}*{Lem.~2.1}. We state it here explicitly since we could not find a suitable reference. 
\begin{lemma}
    \label{lem:independent_real_components}
    Let $S_1$, \dots, $S_r$ be the connected components of $X(\RR)$.
    \begin{itemize}
        \item If $X$ is non dividing, then the fundamental classes $[S_1]$, \dots , $[S_r]$ can be extended to a basis of $H_1(X, \ZZ)$.
        \item If $X$ is dividing, then any subset of $r-1$ fundamental classes of real connected components can be extended to a basis of $H_1(X, \ZZ)$.
    \end{itemize}
\end{lemma}

\begin{definition}
    Let $X$ be a smooth real curve of genus $g$, and let $\sigma$ be the antiholomorphic involution. A basis $(a_1, \dots , a_g ; b_1, \dots , b_g)$ of $H_1(X, \ZZ)$ is called a \emph{Comessatti basis} if:
    \begin{itemize}
        \item $a_i \cdot a_j = b_i \cdot b_j = 0$ and $a_i \cdot b_j = \delta_{ij}$ for $i,j = 1, \dots, g$ (i.e., the basis is symplectic);
        \item $a_i = a_i^\sigma$ for $i=1, \dots g$.
    \end{itemize}
\end{definition}

The existence of such a basis of homology was first proved by Comessatti \cites{comessattiSulleVarietaAbeliane1925a, comessattiSulleVarietaAbeliane1926}. For more recent references, we refer the reader to \cites{silholRealAlgebraicSurfaces1989a,cilibertoRealAbelianVarieties1996a}, where Comesatti bases are called \emph{semireal} and \emph{pseudonormal}, respectively. Our goal in the following is to describe explicit Comessatti bases, using topological models for real curves which can be found e.g.~in \cite{natanzonModuliRealAlgebraic1999}*{Sec.~1}. For a related discussion, see also \cite{vinnikovSelfadjointDeterminantalRepresentations1993}. Hereafter we summarize Natanzon's construction, also to fix the notations which will be used later.

Take an orientable topological surface $Z$ of genus $\tilde{g}$ with $k = m+r$ holes. Equip it with a Riemann surface structure $Z^+$, and consider an atlas of holomorphic charts ${(U_i, z_i)}$, with $Z^+(\CC) = \bigcup U_i$. The conjugate atlas ${(U_i, \overline{z}_i)}$ gives another Riemann surface structure $Z^-$ on $Z$. Consider the natural map $\alpha \colon Z^+ \to Z \to Z^-$, which is antiholomorphic. The complex structures on $Z^+$ and $Z^-$ induce metrics of constant curvature on $Z^+$ and $Z^-$, and $\alpha$ is an isometry with respect to them.

Consider now geodesics $S_1, \dots , S_r, R_1, \dots , R_m$ around the $k = r+m$ holes on $Z^+$. These geodesics are the boundary of a compact orientable surface $\tilde{Z}^+$. Set $\tilde{Z}^- = \alpha(\tilde{Z}^+)$, and identify the boundaries $\partial \tilde{Z}^+ = \bigcup S_i \cup \bigcup R_j $ and $\partial \tilde{Z}^- = \bigcup \alpha(S_i) \cup \bigcup \alpha(R_j)$ as follows:
\begin{itemize}
    \item For $i=1, \dots , r$, identify $S_i$ and $\alpha(S_i)$ by means of $\alpha$.
    \item For $j = 1, \dots , m$, consider isometries $\pi_j \colon R_j \to R_j$ without fixed points such that $\pi_j \circ \pi_j = \mathrm{id}$. Now identify $R_j$ and $\alpha(R_j)$ by means of $\alpha \circ \pi_j$.
\end{itemize}
With these $k=r+m$ identifications we obtain a compact Riemann surface (or, a smooth projective algebraic curve) $X$ of genus $g = 2\tilde{g} +k-1$, see \Cref{fig:real_homology}. The map $\alpha$ induces an antiholomorphic involution $\sigma \colon X(\CC) \to X(\CC)$ such that $X(\RR) = S_1 \sqcup \dots \sqcup S_r$. The real curve $(X, \sigma)$ is dividing if and only if $m=0$, and hence the real curve $(X, \sigma)$ has topological type $(2\tilde{g}+r-1, r, 0)$ if $m = 0$, and $(2\tilde{g}+r+m-1, r, 1)$ if $m \ge 1$ (notice that we use a different notation for the topological type compared to \cite{natanzonModuliRealAlgebraic1999}). We call any curve constructed in this way a \emph{Natanzon model of type $(\tilde{g}, r, m)$}.

Every smooth real curve is topologically equivalent to (possible many) Natanzon models, see \cite{natanzonModuliRealAlgebraic1999}*{Th.~1.1 and 1.2}. This means that, given any such curve $(X, \sigma)$ there exists a model $(X', \sigma')$ and a homeomorphism $\phi \colon X(\C) \to X'(\C)$ such that $\sigma' \circ \phi = \phi \circ \sigma$.

We now describe how the choice of a Natanzon model $X'$ of type $(\tilde{g}, r, m)$ for a curve $X$ over $\RR$ canonically determines a Comessatti basis of $H_1(X, \ZZ)$. To do that, we more precisely describe the action of $\sigma$ on a model of a real curve. In \Cref{fig:real_homology}, outside of the annuli bounded by the gray loops surrounding $R_1, \dots , R_m$, the action of $\sigma$ is just the reflection along a vertical plane passing through $S_1, \dots, S_r, R_1, \dots, R_m$. On the other hand, the action of $\sigma$ inside the annuli bounded by the gray loops is a vertical reflection composed with a \emph{Dehn twist}, see \Cref{fig:antiholomorphic_involution}.

\begin{figure}
    \begin{center}
        \tikzset{every picture/.style={line width=0.75pt}} 
        \includegraphics[width=0.9\linewidth]{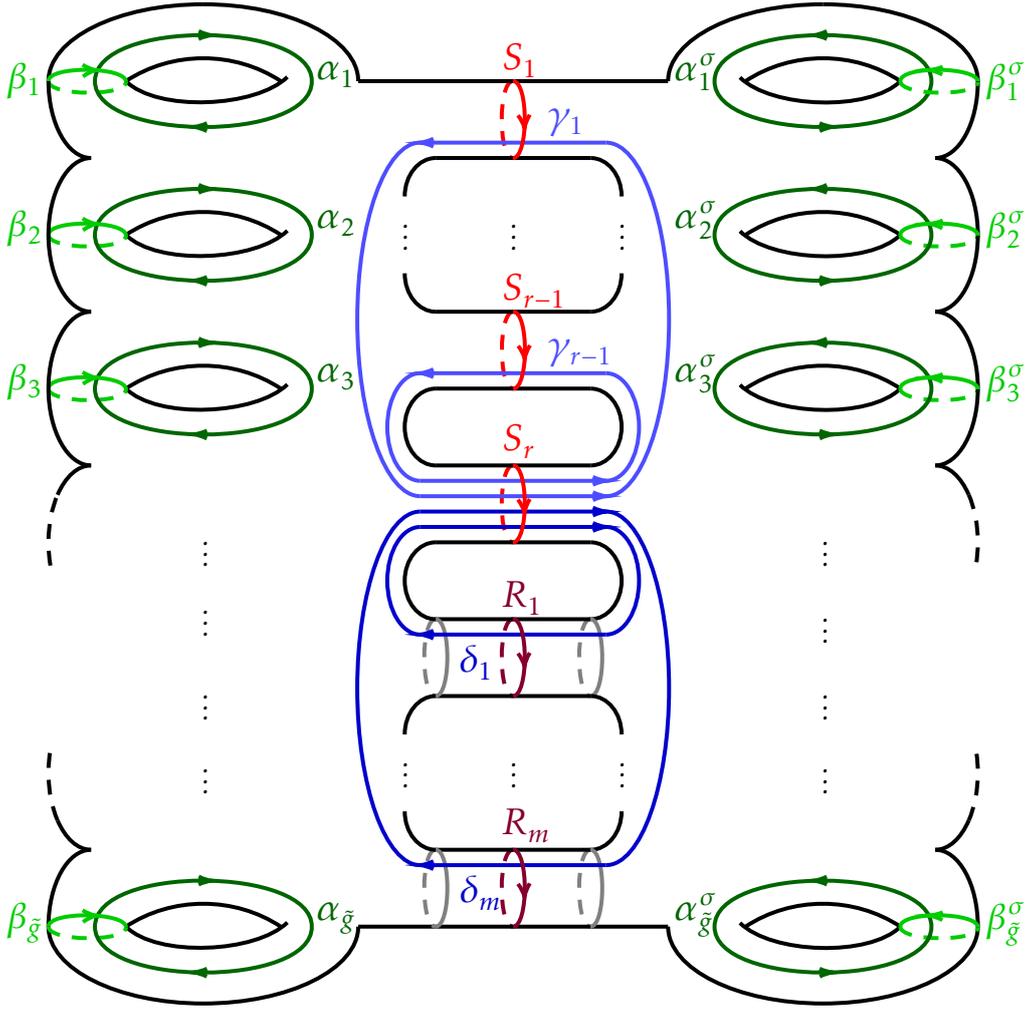}
    \end{center}
    \caption{Pseudonormal basis of homology}
    \label{fig:real_homology}
\end{figure}

\begin{figure}
    \begin{center}
        \tikzset{every picture/.style={line width=0.75pt}} 
        \includegraphics[width=0.9\linewidth]{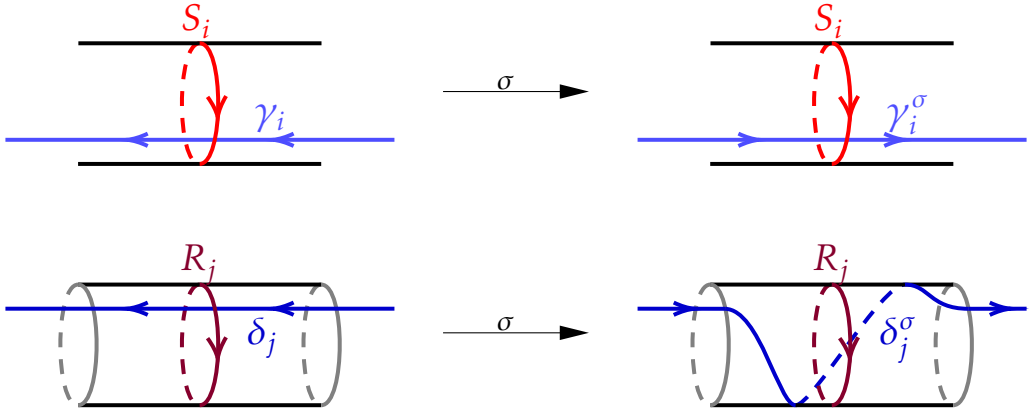}
    \end{center}
    \caption{Action of $\sigma$}
    \label{fig:antiholomorphic_involution}
\end{figure}
We now explicitly describe a Comessatti basis $(a_1, \dots , a_g; b_1, \dots , b_g)$ of $H_1(X, \ZZ)$, for  which we refer again to \Cref{fig:real_homology}.
Let $X$ be a curve over $\RR$ with a Natanzon model $X'$ of type $(\tilde{g}, r , m)$. The model $X'$ is constructed by taking a genus $\tilde{g}$ surface $Z$ with $r+m$ punctures. Take the symplectic basis of homology $([\alpha_1], \dots , [\alpha_{\tilde{g}}]; [\beta_1], \dots , [\beta_{\tilde{g}}])$ for $Z$, see \Cref{fig:real_homology}, constructed with loops not intersecting the geodesics $S_1, \dots , S_r, R_1, \dots , R_m$. These geodesics are always oriented as the boundary of the surface $\tilde{Z}^+$. 

Now assume that $r > 0$, and fix distinct, ordered points $P_1, \dots , P_{r-1}, Q_1, \dots , Q_m$ on the \emph{base loop} $S_r$. Now connect $P_1, \dots , P_{r-1}, Q_1, \dots , Q_m$ respectively with $S_1, \dots, S_r, R_1, \dots, R_m$ with simple curves on the surface $\tilde{Z}^+$, and extend those $r-1+m$ curves with their symmetric images in $\tilde{Z}^+$, see \Cref{fig:antiholomorphic_involution}. In this way, we obtain simple closed loops $\gamma_1, \dots , \gamma_{r-1}, \delta_1, \dots , \delta_{m}$ on $X$. The orientations of these loops are chosen in such a way that $S_i \cdot \gamma_i = R_j \cdot \delta_j = 1$ holds for $i=1, \dots , r-1$ and $j=1, \dots , m-1$.
When $r=0$, i.e., when the curve has no real points, we cannot take $S_r$ as a base loop. But the same construction can be performed taking $R_m$ as a base loop instead of $S_r$, covering also the case $r=0$.

We are now ready to define the basis $(a_1, \dots , a_g ; b_1, \dots , b_g)$ of $H_1(X, \ZZ)$ associated with a Natanzon model of type $(\tilde{g}, r, m)$. We define:
\begin{itemize}
    \item $a_{i} = [\alpha_i +\alpha_i^\sigma]$ for $i = 1, \dots , \tilde{g}$;
    \item $a_{\tilde{g}+i} = [\beta_{\tilde{g}-i+1} + \beta_{\tilde{g}-i+1}^\sigma]$ for $i = 1, \dots , \tilde{g}$;
    \item $a_{2\tilde{g}+i} = [R_{i}]$ for $i = 1, \dots, m$;
    \item $a_{2\tilde{g}+m+i} = [S_{i}]$ for $i = 1, \dots, r-1$;
    \item $b_{i} = [\beta_i]$ for $i = 1, \dots, \tilde{g}$;
    \item $b_{\tilde{g}+i} = [\alpha_{\tilde{g}-i+1}^\sigma]$ for $i = 1, \dots , \tilde{g}$;
    \item $b_{2\tilde{g}+i} = [\delta_{i}]$ for $i = 1, \dots, m$;
    \item $b_{2\tilde{g}+m+i} = [\gamma_{i}]$ for $i = 1, \dots, r-1$.
\end{itemize}
The basis $(a_1, \dots, a_g ; b_1, \dots , b_g)$ is clearly symplectic. Moreover, the action of $\sigma$ on homology in this basis is given by:
\begin{itemize}
    \item $a_{i}^{\sigma} = a_i$ for $i = 1, \dots , 2\tilde{g}+m+r-1 = g$;
    \item $b_i^\sigma = -b_i+a_{2\tilde{g}-i+1}$ for $i = 1, \dots , 2\tilde{g}$;
    \item $b_{2\tilde{g}+i}^\sigma = -b_{2\tilde{g}+i}+a_{2\tilde{g}+i}$ for $i = 1, \dots, m$;
    \item $b_{2\tilde{g}+m+i} = -b_{2\tilde{g}+m+i}$ for $i = 1, \dots, r-1$.
\end{itemize}
and this is a Comessatti basis.

Next, we consider a basis of \emph{real normalized differentials} $\omega_1, \dots, \omega_g$, i.e., such that $\int_{a_i}\omega_j = \delta_{ij}$ and $\omega^\sigma = \omega$. Write $b_i^\sigma = -b_i + \sum _k{m_{ik}} a_k$. Notice that
\[
\overline{\int_{b_i}\omega_j} = \int_{b_i^\sigma}\omega_j^\sigma = - \int_{b_i} \omega_j + \sum_k m_{ik} \int_{a_k} \omega_j = - \int_{b_i} \omega_j + m_{ij} 
\]
which implies that 
$
    2 \mathrm{Re}\left(\int_{b_i}\omega_j\right) = m_{ij} \in \ZZ
$.
Denote by $M = (m_{ij}) \in \ZZ^{g \times g}$ the \emph{reflection matrix}. In terms of the period matrix  $\Pi$ with respect to the bases $\omega_1, \dots , \omega_g$ and $(a_1, \dots, a_g; b_1, \dots , b_g)$, the above computation shows that 
\begin{equation}
    \label{eq:period_matrix}
    \Pi = \left( \begin{array}{ccc|ccc}
        1 & {} & {} & {} & {} & {}\\
        {} & {\ddots} & {} & \multicolumn{3}{c}{\frac{1}{2}M + iT} \\
        {} & {} & {1} & {} & {} & {}
    \end{array} \right)
\end{equation}
with $T$ a real, symmetric, positive definite matrix, thanks to Riemann's bilinear relations, and  $M\in \ZZ^{g\times g}$. We call $\tau = \frac{1}{2}M + iT \in \C^{g\times g}$ the \emph{Riemann matrix}. This special form for the period matrix only requires the basis of $H_1(X, \ZZ)$ to be a Comessatti basis.

Write $I_s$ for the $s\times s$ identity matrix, $\mathbf{0}_{s_1 \times s_2}$ and $\mathbf{1}_{s_1 \times s_2}$ for the $s_1 \times s_2$ matrices with all zeros and ones, respectively, and finally let
\[
    J_s=\overbrace{
        \left(\begin{array}{ccc}
            {} & {} & {1} \\
            {} & \iddots & {} \\
            {1} & & {}
        \end{array}\right)
        }^{s} 
\]
be the $s \times s$ matrix with ones on the principal antidiagonal, and zeros otherwise. If the basis of $H_{1}(X, \ZZ)$ is the one associated with a Natanzon model of type $(\tilde{g},r,m)$, then the reflection matrix has the form
\begin{equation}
    \label{eq:reflection_1}
    M=
    \left(
    \begin{array}{ccc}
        J_{2\tilde{g}} & \mathbf{0}_{2\tilde{g}\times m} & \mathbf{0}_{2\tilde{g}\times (r-1)} \\ \mathbf{0}_{m\times 2\tilde{g}}
         & I_m & \mathbf{0}_{m\times(r-1)} \\
         \mathbf{0}_{(r-1)\times 2\tilde{g}} & \mathbf{0}_{(r-1)\times m} & \mathbf{0}_{(r-1)\times (r-1)}
    \end{array}\right)
\end{equation}
with base loop $S_r$, while it has the form
\begin{equation}
    \label{eq:reflection_2}
    M=
    \left(
    \begin{array}{ccc}
    J_{2\tilde{g}} & \mathbf{0}_{2\tilde{g}\times (m-1)} & \mathbf{0}_{2\tilde{g}\times r} \\ 
    \mathbf{0}_{(m-1)\times2\tilde{g}} & \mathbf{1}_{(m-1)\times(m-1)}-I_{m-1} & \mathbf{1}_{(m-1)\times r} 
        \\ 
        \mathbf{0}_{r\times 2\tilde{g}} & \mathbf{1}_{(m-1)\times r} & \mathbf{1}_{r \times r}
    \end{array}\right)
\end{equation}
with base loop $R_m$.

We have then proved the following theorem, which is essentially due to Comessatti. Our contribution lies in its proof, based on the explicit construction of models for real curves, and on the explicit choice of the Comessatti basis of homology.

\begin{theorem}[Comessatti]\label{thm:excombasis}
    Let $X$ be a smooth real curve. Then there exists a Comessatti basis of $H_1(X, \ZZ)$. In particular, given any Natanzon model of type $(\tilde{g}, r, m)$ for $X$, the associated basis of homology is a Comessatti basis, and the reflection matrix $M$ takes the special form \eqref{eq:reflection_1} or \eqref{eq:reflection_2}, respectively if the base loop is $S_r$ or $R_m$. 
\end{theorem}

Many properties of the Jacobian $J=\jac(X)$ of a real curve can be directly deduced from the explicit description of the period and reflection matrices given above.
For instance, one can deduce the following corollary (see also \cite{grossRealAlgebraicCurves1981a}*{Prop. 1.1, 3.2, and 3.3} for a proof that does not make use of the special form of $M$).

\begin{corollary}[Comessatti]
    \label{cor:comessatti}
    Let $X$ be a smooth real curve and $J = \jac(X)$ its Jacobian.
    \begin{enumerate}
        \item If $X(\RR) \neq \emptyset$, then $J(\RR) \cong (\RR/\ZZ)^g\times (\ZZ/2)^{r-1}$ (as a Lie group).
        \item If $X(\RR) = \emptyset$ and $g$ even, then $J(\RR) \cong (\RR/\ZZ)^g$.
        \item If $X(\RR) = \emptyset$ and $g$ odd, then $J(\RR) \cong (\RR/\ZZ)^g\times (\ZZ/2)$.
    \end{enumerate}
\end{corollary}

\subsection{Volumes of principally polarized real abelian varieties}
\label{sec:volume}
We follow \cite{milneAbelianVarieties1986} to introduce principally polarized abelian varieties over $\RR$, and \cites{birkenhakeComplexAbelianVarieties2004,griffithsPrinciplesAlgebraicGeometry1978} for classical complex ones. Our definition is equivalent to the one used in \cites{silholRealAlgebraicSurfaces1989a,cilibertoRealAbelianVarieties1996a}.

An \emph{abelian variety over $\RR$} is an irreducible complete group variety over $\RR$. In particular, if $A$ is an abelian variety over $\RR$, then $A(\RR) \neq \emptyset$ since $0\in A(\RR)$. 
A \emph{principal polarization} on an abelian variety over $\RR$ is an isogeny $f \colon A \to A^\vee$ of degree $1$ such that its complexification $f_\C = \varphi_{\mathcal{L}}$ is the canonical map associated to some invertible sheaf $\mathcal{L}$ on $A_\C$.
If $A$ is an abelian variety over $\RR$ and $f$ is a principal polarization, we call $(A, f)$ a \emph{principally polarized real abelian variety}.

Let $(A, f)$ be a principally polarized real abelian variety. 
Write $A_{\C} \cong V / \Lambda$ and $A_{\C}^\vee \cong \hom_{\bar{\C}}(V, \C)/\Lambda^\vee$, where $W\coloneqq\hom_{\bar{\C}}(V, \C)$ denotes the $\C$-vector space of $\C$-antilinear functions. Then $f_\C \colon V / \Lambda \to W/\Lambda^\vee$, and there exists a unique $\C$-linear map $F \colon V \to W$ with $F(\Lambda) \subset \Lambda^\vee$ inducing $f_\C$ (see e.g.~\cite{birkenhakeComplexAbelianVarieties2004}*{Prop.~1.2.1}). Furthermore, the Hermitian form $H \colon V \times V \to \C$, $(z,w) \mapsto F(z)(w)$ is positive definite \cite{birkenhakeComplexAbelianVarieties2004}*{Th.~2.5.5 and Sec.~4.1}. The inverse $H^{-1}$ is also a positive definite Hermitian form. By parallel transport, the positive definite Hermitian form $H^{-1}$ on $V$, identified with the holomorphic tangent space of $A_\C$ at the origin, defines a translation invariant Hermitian (K\"ahler) metric $\dd s^2$ on $A_\C$ \cite{griffithsPrinciplesAlgebraicGeometry1978}*{p.~301}. Recall also that the real part of an Hermitian metric defines a Riemannian metric on the underlying real manifold \cite{griffithsPrinciplesAlgebraicGeometry1978}*{p.~28}.

\begin{definition}
    Let $(A, f)$ be a principally polarized real abelian variety. We call the translation invariant Hermitian metric $\dd s^2$ defined above the \emph{canonical Hermitian metric} of $(A, f)$, and $\mathrm{Re(\dd s^2)}$ the \emph{canonical Riemannian metric} of $(A, f)$.
\end{definition}

In the following, when referring to metric properties of (real) submanifolds of $A_\C$, we will always refer to the canonical Riemannian metric. 
In particular, the volume $\vol(A(\RR))$ is well defined because $A(\RR)$ is a $(\dim A)$-dimensional real submanifold of $A_\C(\C)$.

Generalizing the case of Jacobians of curves described in \Cref{sec:comessatti}, there exists 
a notion of \emph{Comessatti form} for the period matrix of a principally polarized real abelian variety $(A, f)$. This result is again due to Comessatti; for modern references we point the reader to \cites{silholRealAlgebraicSurfaces1989a, cilibertoRealAbelianVarieties1996a}. More precisely, the period matrix $\Pi$ can be taken as in \eqref{eq:period_matrix}, with $M \in \ZZ^{\dim A \times \dim A}$ symmetric and $T$ symmetric, real and positive definite. 
\begin{theorem}
    \label{thm:volume_identity}
    Let $(A, f)$ be a principally polarized abelian variety over $\RR$ and $A(\RR)_0$ the real connected component of the identity. Then
    \[
        \vol(A(\RR)_0) = (\det T)^{-\frac{1}{2}}
    \]
    where the volume is computed with respect to the canonical Riemannian metric of $(A, f)$, and $T$ is the imaginary part of a Riemann matrix for $(A, f)$ in the form \eqref{eq:period_matrix}. In particular, $\det T$ is an invariant of $(A, f)$, i.e., it does not depend on the choice of the period matrix.
\end{theorem}
\begin{proof}
    Let $g = \dim A$.
    If $A \cong V / \Lambda$, where $\Lambda$ are the columns of the period matrix $\Pi$ as in \eqref{eq:period_matrix}, then $e_1, \dots , e_g$ generate $V$, and it follows e.g.~from \cite{birkenhakeComplexAbelianVarieties2004}*{Prop.~8.1.1} that $T^{-1} = \mathrm{Im}(\tau)^{-1}$ is the Gram matrix of the Hermitian form $H^{-1}$ induced by the principal polarization $f$ with respect to the basis $e_1, \dots , e_g$. Thanks to Riemann's bilinear relations, $T$ is real, symmetric, positive definite, and we can write $T = B^t B$ for some $B \in \RR^{g\times g}$ such that $\det B >0$. 

    Set $u_i = B^{-1} e_i \in \RR^g$, so that $u_1, \dots , u_g$ is an orthonormal basis for $H^{-1}$. 
    If $z_1, \dots , z_g$ are the corresponding complex coordinates and $x_1, \dots , x_g$ are the real coordinates corresponding to $e_1, \dots , e_g$,
    then the canonical Hermitian metric is given by $\omega = \dd z_1 \dd \overline{z}_1 + \dots + \dd z_g \dd \overline{z}_g$ and we have 
    \[
        \vol(A(\RR)_0) = \int_{A(\RR)_0} \dd \vol_\omega =  \int_{[0,1]^g} \det B^{-1} \ \dd x_1 \cdots \dd x_g = \det B^{-1} = \det T ^{-\frac{1}{2}}
    \]
    where $\dd \vol_\omega$ denotes the volume form with respect to the canonical Riemannian metric $\mathrm{Re}(\omega)$.
\end{proof}

It is easy to deduce from \Cref{thm:volume_identity} how to read the full volume of $A(\RR)$ from the Riemann matrix.
\begin{corollary}
    \label{cor:volume}
    We follow the same notations of \Cref{thm:volume_identity} and \eqref{eq:period_matrix}. Let $g = \dim A$ and $\gamma = \rank_{\ZZ/2} M$ the rank of the reflection matrix $M$ modulo $2$. 
    Then $A(\RR)$ has $2^{g-\gamma}$ connected components and:
    \[
        \vol(A(\RR)) = 2^{g-\gamma}(\det T)^{-\frac{1}{2}}
    \]
\end{corollary}
\begin{proof}
 The connected components of $A(\RR)$ are translations of the identity component $A(\RR)_0$. Since the canonical Riemannian metric is translation invariant and by \Cref{thm:volume_identity}, it suffices to prove that $A(\RR)$ has $2^{g-\gamma}$ connected components. A point in $A(\CC)$ represented by $x+iy \in\C^g$, $x,y\in\RR^g$, is fixed by the involution if and only if $2iy=k+(\frac{1}{2}M+iT)l$ for some $k,l\in\Z^g$. Thus the set of connected components of $A(\RR)$ is in bijection with classes $\bar{l}\in(\ZZ/2)^g$ such that $Ml\in2\Z^g$.
\end{proof}

Notice the difference between the complex and the real picture: with respect to the canonical Riemannian metric, $\vol(A(\RR))$ is given in \Cref{cor:volume}, while for the complex points we simply have $\vol(A_\CC(\CC)) = 1$.

\subsection{Metric properties of real Jacobians}
\label{sec:Bergman}

We start by adapting the results of \Cref{sec:volume} to the Jacobian $J =\jac(X)$ of a smooth real curve $X$. The Jacobian has a canonical principal polarization given by intersection of cycles. The positive definite Hermitian form $H^{-1}$ on $H^0(X, \Omega)^*$ induced by the principal polarization is the dual to the Hodge product on holomorphic $1$-forms, defined as \[\inner{\omega_1}{\omega_2} \coloneqq \frac{i}{2}\int_X \omega_1 \wedge \overline{\omega}_2 = \frac{1}{2}\int_X \omega_1 \wedge \star\overline{\omega}_2\] for $\omega_1, \omega_2 \in H^0(X, \Omega)$. We refer the reader e.g.~to \cite{griffithsPrinciplesAlgebraicGeometry1978}*{p.~232}. In particular, if $\Pi$ is a period matrix for $X$ in as in \eqref{eq:period_matrix}, then $T_{ij} = \inner{\omega_i}{\omega_j}$.
Finally, recall that the Abel--Jacobi map
$
    \phi \colon X \to \jac(X)
$
is an embedding.
\begin{definition}
    The pull-back metric on $X$ via the Abel--Jacobi map $\phi$ of the canonical metric on $\jac(X)$ is called the  \emph{Bergman} or \emph{canonical} (Hermitian) metric on $X$. We call its real part the \emph{Bergman Riemannian metric} on $X$.    
\end{definition}
 Notice that, since the canonical metric on $\jac(X)$ is translation invariant, the Bergman metric does not depend on the choice of the base point $P_0$ for the Abel--Jacobi map. For more properties of this metric, we refer the reader to \cites{wentworthAsymptoticsArakelovGeensFunction1991,habermannRiemannianMetricsTeichmuller1996a}.

Explicitly, the Bergman metric can be described as follows (compare with the proof of \Cref{thm:volume_identity}). Fixing a symplectic basis of cycles and a normalized basis of holomorphic $1$-forms $\omega_1, \dots, \omega_g$, recall that the Gram matrix $T =(\inner{\omega_i}{\omega_j})_{i,j}$ in the basis $\omega_1, \dots, \omega_g$ for the Hodge inner product is the imaginary part of the Riemann matrix. Therefore, if we write $T = B^t B$, an orthonormal basis $\theta_1, \dots, \theta_g$ for $H^0(X, \Omega)$ with respect to the Hodge product can be defined by $\theta_j = B^{-1}\omega_j$. The Bergman Riemannian metric is therefore $\sum_{j=1}^g \theta_j \overline{\theta}_j$. Using this metric, the volume of $\jac(X)(\C)$ is equal to $1$, while the volume of $X$ is $g$. In the literature, a version with the normalization constant $1/g$ is also used.

We now assume that $X(\RR)\neq \emptyset$. Then we can choose $P_0 \in X(\RR)$, and the Abel--Jacobi map $\phi$ is defined over $\RR$. In particular, in this case we have $\phi(X(\RR)) \subset \jac(X)(\RR)$.
In the following, we will make a quantitative comparison between $\phi(X(\RR))$ and $\jac(X)(\RR)$, with the ultimate goal of studying the totally real divisor threshold. In this direction, we can finally adapt the volume computation of \Cref{sec:volume} to real Jacobians.

\begin{proof}[{Proof of \Cref{thm:D}}]
By \Cref{cor:comessatti}, the number of connected components of $\jac(X)(\RR)$ is $2^{r-1}$. Then the claim follows from \Cref{cor:volume}.
\end{proof}

\section{Qualitative analysis: real algebra}
\label{sec:qualitative}
This section is dedicated to the proofs of \Cref{thm:B,thm:C}. 
For the entire section we fix a real closed field $R_1$ and denote by $R_2$ the field of algebraic Puiseux series over $R_1$.
We will study families of curves over $R_1$ to obtain results about curves over $R_2$, and vice versa.

\subsection{Families of curves}
\label{sec:families}
For $g\geq0$, $d\geq 2g+1$ and $n=d-g$ let $\cH$ be the Hilbert scheme of curves in $\PP^n$ of genus $g$ and degree $d$, and let $Z \subset \PP^n \times \cH$ be the universal family. 
We denote by $\cH'$ the set of curves which are not contained in a hyperplane.
Let $\cH_{(g,r,a)} \subset \cH'(R_1)$ denote the subset of $h \in \cH'(R_1)$ such that $Z_h = \pi_1(\pi_2^{-1}(h)) \subset \PP^n_{R_1}$ is a smooth curve of topological type $(g, r, a)$. It follows from Hardt's triviality theorem \cite{basuAlgorithmsRealAlgebraic2006}*{\S5.8} that this is a semialgebraic set.
By the Riemann--Roch theorem, each such curve is embedded via a complete linear system since the $Z_h$ are assumed not to be contained in a hyperplane.
We are interested in studying the subset \[\cH_{(g,r,a)}(m) = \{ \, h \in \cH_{(g,r,a)} \mid \N(Z_h) \le m  \} .\]
\begin{lemma}\label{lem:bound_nx_semialgebraic}
    The set $\cH_{(g,r,a)}(m)$ is semialgebraic for all $m\in\N$. 
\end{lemma}
\begin{proof}
    Let $k \in \NN$ be such that $k d-m \ge g$. Then, for all divisors $D_1$ and $D_2$ of respective degrees $d$ and $m$, the divisor $kD_1 - D_2$ is linearly equivalent to an effective divisor.
    Now notice that $\N(Z_h) \le m$ if and only if for all effective divisors $E$ of degree $kd-m$ on $Z_h$, there exists a homogeneous polynomial $F$ of degree $k$ whose divisor on $Z_h$ is of the form $F+F'$ where $F'$ is totally real and effective. This is a semialgebraic condition and hence $\cH_{(g,r,a)}(m)$ is semialgebraic. 
\end{proof}

Recall that $R_2$ denotes the field of algebraic Puiseux series over $R_1$ and let $h\colon (0,1]\to \cH_{(g,r,a)}$ be a  semialgebraic path. Its germ corresponds to a point $h_+\in \cH(R_2)$ for which $X_{R_2}\coloneqq Z_{h_+}$ is a smooth curve over $R_2$. Furthermore, for $\epsilon\in(0,1]$, let $X_\epsilon\coloneqq Z_{h(\epsilon)}$.

\begin{lemma}\label{lem:unboundediffinf}
 With $X_\epsilon$ and $X_{R_2}$ as above, we have $\N(X_{R_2})=\infty$ if and only if $\N(X_\epsilon)$ is unbounded for $\epsilon \to 0$. 
\end{lemma}
\begin{proof}
    Let $m\in\NN$. By \Cref{lem:bound_nx_semialgebraic} and \cite{basuAlgorithmsRealAlgebraic2006}*{Prop.~3.17} there exists $\epsilon_0>0$ such that $\N(X_\epsilon)\leq m$ for all $\epsilon\in(0,\epsilon_0)$ if and only if $\N(X_{R_2})\leq m$. This implies the claim.
\end{proof}

\subsection{Curves with not so many components}
\label{sec:few_ovals}
In this section we prove the key result of the paper, namely that if the number of connected components of a curve is less than $g$, then the totally real divisor threshold cannot be bounded from above just in terms of the invariants $g$, $r$ and $a$ of the curve.

Let $A=R_1\llbracket t\rrbracket^{\mathrm{alg}}$ be the ring of algebraic formal power series over $R_1$. Then the field of algebraic Puiseux series $R_2$ is the real closure of the field of fractions of $A$. Let $X$ be a flat family over $A$ whose base change $X_{R_2}$ is a smooth and geometrically irreducible curve of genus $g\geq2$ over $R_2$ and whose special fiber $X_{R_1}$ is a stable curve of genus $g$. Our first goal is to prove the following.
\begin{theorem}
    \label{thm:criterioninf}
    With the above notations, let $\widetilde{X}_{R_1}\to X_{R_1}$ be the normalization of the special fiber $X_{R_1}$ of $X$.
\begin{enumerate}
    \item If $\widetilde{X}_{R_1}({R_1})=\emptyset$ and the divisor class group of $\widetilde{X}_{R_1}$ is not finitely generated, then $\N(X_{R_2})=\infty$.
    \item If $X_{R_1}$ is smooth, $m\geq 2g+1$ and $\N(X_{R_2})\leq m$, then $\N(X_{R_1})\leq m$.
\end{enumerate}
\end{theorem}
\begin{proof}
    It follows from our assumptions and \cite{egaIV3}*{15.7.10, 15.7.8} that $X$ is proper over $A$. Recall from \cite{delignemumford}*{\S1} that there is a canonical invertible sheaf $\omega_{X/A}$ on $X$ such that $\omega_{X/A}\otimes_A R_2$ is the canonical sheaf of $X_{R_2}$ and such that for all $n\geq3$ the sheaf $\omega_{X/A}^{\otimes n}$ is relatively very ample with $H^0(X,\omega_{X/A}^{\otimes n})\cong A^M$, where $M=(2n-1)(g-1)$. Thus we obtain an embedding $X\to\PP_A^M$ for which the induced embedding $X_{R_2}\to\PP^M_{R_2}$ is the $n$-th pluricanonical map. We also obtain a morphism $f\colon\widetilde{X}_{R_1}\to\PP^M_{R_1}$ whose image is $X_{R_1}$. In case $(ii)$, the morphism $f$ is also an embedding since $\widetilde{X}_{R_1}={X}_{R_1}$. 

    Now let $m\geq 2g+1$ and assume that $\N(X_{R_2})\leq m$. In case $(i)$ we will produce a contradiction and in case $(ii)$ we will show that $\N(X_{R_1})\leq m$.    
    Let $S\subset\widetilde{X}_{R_1}$ be the preimage of the singular locus of $X_{R_1}$ under the normalization map.     
    Let $D$ be a reduced effective divisor on $\widetilde{X}_{R_1}$ of degree $d\geq m$. In case $(ii)$, we note that every divisor class of degree $d\geq m$ has such a representative because $m\geq 2g+1$. In case $(i)$, we additionally assume that the support of $D$ is disjoint from $S$ and that its class $[D]$ is not contained in the subgroup of the divisor class group generated by classes of points from $S$ . Such a divisor exists when $d$ is large enough by our assumption that the divisor class group of $\widetilde{X}_{R_1}$ is not finitely generated.
    
    Choosing $n$ large enough, there is a hyperplane $H\subset\PP^M_{R_1}$  that contains the image of the support of $D$ under $f\colon\widetilde{X}_{R_1}\to\PP^M_{R_1}$ and intersects $X_{R_1}$ transversally. By Hensel's lemma, for every point in the support of $D$ there is a point in $H\cap X_{R_2}$ that specializes to it. Hence the divisor on $X_{R_2}$ defined by $H$ can be written as the sum $E=E_1+E_2$ of two effective divisors such that $E_1$ specializes to $f(D)$ and $E_2$ to the sum of the remaining intersection points of $X_{R_1}$ with $H$. We denote the divisor on $\widetilde{X}_{R_1}$ consisting of these remaining intersection points by $D_2$, i.e., the divisor defined by $H$ on $\widetilde{X}_{R_1}$ is $D+D_2$. Since $\deg(E_1)\geq \N(X_{R_2})$, $E_1$ is linearly equivalent to a totally real divisor $E_1'$ on $X_{R_2}$. There is a hyperplane $H'$ of $\PP^M_{R_2}$ whose divisor on $X_{R_2}$ is $E_1'+E_2$. Every point of $E_1'$ specializes to a real point of $X_{R_1}$. Thus, in case $(i)$, the hyperplane $H'$ specializes to a hyperplane whose divisor on $\widetilde{X}_{R_1}$ is of the form $D_1'+D_2$, where $D_1'$ is an effective divisor whose support is contained in $S$ ($X_{R_1}$ has only singular real points by the assumption  $\widetilde{X}_{R_1}(R_1)=\emptyset$). Because $D$ is linearly equivalent to $D_1'$, this is a contradiction to our assumtion that $[D]$ is not contained in the subgroup generated by classes of points of $S$. In case $(ii)$, the hyperplane $H'$ specializes to a hyperplane whose divisor on ${X}_{R_1}$ is of the form $D_1'+D_2$, where $D_1'$ is an effective and totally real divisor. Because $D$ is linearly equivalent to $D_1'$, this shows that $\N(X_{R_1})\leq d$. Since in case $(ii)$ we can choose $d=m$, this shows $\N(X_{R_1})\leq m$.
\end{proof}
We now give examples of families of curves over ${R_1}$, explicit for almost all topological types, where \Cref{thm:criterioninf} applies. We start with hyperelliptic families, for which we refer the reader to \cite{grossRealAlgebraicCurves1981a}*{Sec.~6}. 

\begin{example}\label{ex:nondiv}
    Let $g\geq2$ and $0\leq r< g$. We construct a curve of topological type $(g,r,1)$.
    Let $a_1<\cdots<a_r$ be elements of ${R_1}$ and consider pairwise distinct $b_1,\ldots,b_{g+1-r},c_1,\ldots,c_{g+1-r}\in {R_1}(i)$  such that $b_i = \overline{c}_i$ is the conjugate of $c_i$ for all $i=1,\ldots,g+1-r$. The hyperelliptic curve 
    \begin{equation*}
        y^2=-\prod_{i=1}^r(x-a_i+t)\cdot\prod_{i=1}^r(x-a_i-t)\cdot\prod_{i=1}^{g+1-r}(x-b_i)\cdot\prod_{i=1}^{g+1-r}(x-c_i)
    \end{equation*}
    of genus $g$ over $A$ satisfies the assumptions of \Cref{thm:criterioninf}. Indeed, the special fiber $X_{R_1}$ has $r$ nodes. All of them are isolated real points and its normalization $\widetilde{X}_{R_1}$ is an irreducible curve of genus $g-r$ without ${R_1}$-points. It follows from \Cref{thm:criterioninf} that $\N(X_{R_2})=\infty$. The set of $R_2$-points of $X_{R_2}$ has $r$ connected components, and it is not of dividing type.
\end{example}
\begin{example}\label{ex:divodd}
    Let $g\geq2$ and $1\leq r< g$. We construct a curve of topological type $(2g-1,2r,0)$.
    To this end, consider the unramified double cover $Y_{R_2}$ of the curve $X_{R_2}$ from \Cref{ex:nondiv} obtained by adjoining a square root $z$ of
    \begin{equation}\label{eq:sqr}
     -\prod_{i=1}^r\frac{x-a_i+t}{x-a_i-t}
    \end{equation}
    to the function field of $X_{R_2}$. By the Riemann--Hurwitz formula, the genus of $Y_{R_2}$ is $2g-1$ where $g$ is the genus of $X_{R_2}$. 
    Because the expression in \Cref{eq:sqr} has an even number of double roots on each of the $r$ connected components of $X_{R_2}(R_2)$, the number of connected components of $Y_{R_2}(R_2)$ is $2r$. 
    Moreover, the map $Y_{R_2}\to\mathbb{P}^1_{R_2}$ defined by $z$ is real-fibered in the sense that if an $R_2(i)$-point of $Y_{R_2}$ is mapped to an $R_2$-point, then it must be an $R_2$-point itself. This implies that $Y_{R_2}$ is of dividing type. Finally, we want to apply \Cref{thm:criterioninf} to show that $\N(Y_{R_2})=\infty$. To this end, we first note that by \cite{fultonNotesStableMaps1997}*{Prop.~6} (``properness of the Deligne--Mumford stack of stable maps'') there is a flat family $Y$ over ${R_1}\llbracket t^{1/n}\rrbracket$, for some $n\in\N$, whose base change to $R_2$ is $Y_{R_2}$ and whose special fiber $Y_{R_1}$ is a stable curve that maps onto the special fiber $X_{R_1}$ from \Cref{ex:nondiv}. This induces a map $\widetilde{Y}_{R_1}\to\widetilde{X}_{R_1}$ of normalizations. In particular, the set of ${R_1}$-points of $\widetilde{Y}_{R_1}$ is empty and the divisor class group of $\widetilde{Y}_{R_1}$ is not finitely generated because both are true for $\widetilde{X}_{R_1}$. Thus by \Cref{thm:criterioninf} we have $\N(Y_{R_2})=\infty$.
\end{example}
\begin{example}\label{ex:diveven}
    Let $g\geq2$ and $1\leq r< g$. We construct a curve of topological type $(2g,2r-1,0)$. To this end,
    we consider a small modification of \Cref{ex:divodd}, namely the branched double cover $Y'_{R_2}$ of the curve $X_{R_2}$ from \Cref{ex:nondiv} obtained by adjoining a square root $z$ of
    \begin{equation}\label{eq:sqr2}
     -\frac{x-a_1+\frac{t}{2}}{x-a_1-t}\cdot\prod_{i=2}^r\frac{x-a_i+t}{x-a_i-t}
    \end{equation}
    to the function field of $X_{R_2}$. This time, the genus of $Y'_{R_2}$ is $2g$ and the number of connected components of $Y'_{R_2}(R_2)$ is $2r-1$. 
    Again, the map $Y'_{R_2}\to\mathbb{P}^1_{R_2}$ defined by $z$ is real-fibered which implies that $Y'_{R_2}$ is of dividing type. The same argument as in \Cref{ex:divodd} shows that $\N(Y'_{R_2})=\infty$ if $r<g$.
\end{example}

The preceding three examples cover all the topological types except for dividing $(M-2)$-curves in the even genus case, i.e., the cases with topological types $(2k, 2k-2,0)$.
These remaining cases are covered by the following example, which is however not at the same level of concreteness as the previous ones.
\begin{example}\label{ex:div}
 Consider a stable curve over ${R_1}$ whose base change to ${R_1}(i)$ consists of two smooth curves of genus $k\geq1$ that intersect transversally in $r>0$ points and that are exchanged by the non-trivial ${R_1}$-automorphism of ${R_1}(i)$. 
 The genus of such a curve is $g=2k+r-1$. 
 We will argue that there exists a flat family over $A$ having this curve as special fiber and whose base change to $R_2$ is a smooth curve $X_{R_2}$ of diving type whose set of $R_2$-points has $r$ components. Then $\N(X_{R_2})=\infty$ by \Cref{thm:criterioninf}.
 For $k=1$, this covers the case of dividing $(M-2)$-curves.
 It appears to be well-known among experts that it is possible to find such a family but for lack of a reference we provide a brief argument.

 First, assume there is some flat family $X$ over $A$ whose base change $X_{R_2}$ is a smooth projective and geometrically irreducible curve and whose special fiber $X_{R_1}$ is a stable curve as above. Let $S$ be the set of singular points of $X_{R_1}$. 
As in \cite{hartshorneDeformationTheory2010}*{Excercise~5.7} the induced family over the dual numbers gives a global section $\alpha$ of the Lichtenbaum--Schlessinger sheaf $\mathcal{T}^1_{X_{R_1}}=\mathcal{T}^1\left(X_{R_1}/{R_1},\mathcal{O}_{X_{R_1}}\right)$. The space $H^0(X,\mathcal{T}^1_{X_{R_1}})$ can be identified with $({R_1})^S$ in a way that for every point $s\in S$, where $\alpha\in ({R_1})^S$ takes a positive value, there is an $R_2$-point of $X_{R_2}$ that specializes to $s$. We have to show that a family $X$, for which $\alpha$ takes only positive values, exists. As there is no $H^2$ on a curve, there exists such a family over the dual numbers by the exact sequence from \cite{hartshorneDeformationTheory2010}*{Ex.~5.7}. Finally, such a family over the dual numbers can be lifted to one over $A$ because an embedding of $X_{R_1}$ to some $\PP^n_{R_1}$ via a complete linear system of large enough degree is a smooth point of the Hilbert scheme. Indeed, as shown in the proof of \cite{hartshorneDeformationTheory2010}*{Prop.~29.9}, for such an embedding one has $H^1(Y,\mathcal{N}_{Y/\PP^n})=0$ which implies smoothness of the Hilbert scheme at this point by \cite{hartshorneDeformationTheory2010}*{Theorem 1.1.c}.
\end{example}

Now we are ready to prove the first two of our main results. We start with a refined version of \Cref{thm:A}.

\begin{theorem}
    \label{thm:A_complete}
    Let $R$ be a real closed field.
    \begin{enumerate}
        \item If $X$ is a smooth curve over $R$ of genus $g$ with $\N(X)=\infty$ such that $X(R)$ has $r>0$ (semialgebraically) connected components, then $r<g$ and $R$ is non-Archimedean.
        \item Conversely, if $R$ is non-Archimedean and $(g,r,a)$ is an admissible topological type with $0\leq r<g$, then there exists a smooth curve $X$ over $R$ with topological type $(g,r,a)$ such that $\N(X)=\infty$.
    \end{enumerate}
    \end{theorem}
\begin{proof}
Let $R$ be a real closed field and $X$ a smooth curve over $R$ of genus $g$. If $R$ is Archimedean, then, by H\"older's theorem \cite{ERA}*{Theorem 2.1.10}, there is an embedding of $R$ to $\RR$. By Scheiderer's result \cite{scheidererSumsSquaresRegular2000}, we have $\N(X_{\RR})\leq m$ for some $m\in\NN$ where $X_{\RR}$ is the base change of $X$ to $\RR$. Since this is a semialgebraic condition by \Cref{lem:bound_nx_semialgebraic}, Tarski's principle shows that $\N(X)\leq m$. If $r\geq g$, then the results of Huisman \cite{huismanGeometryAlgebraicCurves2001} and Monnier \cite{monnierDivisorsRealCurves2003} imply that $\cH_{(g,r,a)}=\cH_{(g,r,a)}(2g-1)$ holds over $\RR$. Thus, by Tarski's principle, one has $\N(X)\leq 2g-1$ for every smooth curve $X$ of genus $g$ over any real closed field $R$ whenever $X(R)$ has at least $r$ components. This shows part (i) of \Cref{thm:A}.

Now let $R$ be a non-Archimedean real closed field.  Let $R_1$ be the relative algebraic closure of $\Q$ in $R$, which is a real closed field. Because $R$ is not Archimedean, there exists $t\in R$ such that $0<t<\frac{1}{n}$ for all $n\in\NN$. Such $t$ is necessarily transcendental over $R_1$. Therefore, the relative algebraic closure of $R_1(t)$ in $R$ can be identified with the field $R_2$ of algebraic Puiseux series over $R_1$.

Consider an admissible topological type $(g,r,a)$ with $0\leq r<g$ and choose a smooth curve $X_{R_2}$ over $R_2$ of this topological type  with $\N(X_{R_2})=\infty$. If $r=0$, any such curve satisfies $\N(X_{R_2})=\infty$ by definition; if $r>0$ and $a=1$, then we can choose $X_{R_2}$ as in \Cref{ex:nondiv}; if $r>0$ and $a=0$, then $g=2k+r-1$ for some $k\geq1$ and we can choose $X_{R_2}$ as in \Cref{ex:div}.
By \Cref{lem:bound_nx_semialgebraic} and Tarski's principle, this shows that $\N(X_{R})=\infty$ where $X_R$ is the base change of $X_{R_2}$ to $R$. This shows part (ii) of \Cref{thm:A}.
\end{proof}

{To prove \Cref{thm:B}, we need to show how the totally real divisor threshold behaves on proper flat families.}
\begin{theorem}[Semicontinuity]
    \label{thm:semicontinuity}
    Let $T$ be a {scheme that is locally of finite type} over $R_1$ and $X\to T$ a proper flat family of smooth curves of genus $g\geq2$ over $T$. For every $m\geq 2g+1$, the set
    \begin{equation*}
        \N(T,m)=\{\, t\in T(R_1)\mid \N(X_t)\leq m \,\}
    \end{equation*}
    is closed in the Euclidean topology on $T(R_1)$.
\end{theorem}
\begin{proof}
    By \cite{delignemumford}*{\S1}, the morphism $X\to T$ factors through an embedding of $X$ to $\PP^n_T$, $n=5g-6$, such that the induced embedding of each $X_t$, $t\in T$, is the third pluricanonical embedding, cf.~the proof of \Cref{thm:criterioninf}. Denote by $f\colon T\to \cH$ the induced map to the Hilbert scheme $\cH$ of curves of genus $g$ and degree $6g-6$ in $\PP^n$. Then $\N(T,m)$ is the preimage under $f$  of  $\cup \cH_{(g,r,a)}(m)$, where the union is over of all topological types of genus $g$ and $\cH_{(g,r,a)}(m)$ is the set defined in \Cref{sec:families}. Thus, by \Cref{lem:bound_nx_semialgebraic}, the set $\N(T,m)$ is semialgebraic. 
   
   Let $t_0$ be in the closure of $\N(T,m)$. We have to show that $t_0\in \N(T,m)$. By the curve selection lemma{, which applies as $T$ is locally of finite type}, there is a semialgebraic path
   \begin{equation*}
       \alpha\colon (0,1)\to \N(T,m)
   \end{equation*}
   with $\lim_{\epsilon\to0}\alpha(\epsilon)=t_0$. Without loss of generality, we may assume that $\alpha$ is not constant. Let $C\subset T$ be the Zariski closure of the image of $\alpha$ and $\pi\colon\tilde{C}\to C$ the normalization. By the existence of semialgebraic sections \cite{ScheidererBook}*{Proposition 4.5.9}, there exists $\epsilon_0>0$ and a semialgebraic path $\beta\colon(0,\epsilon_0)\to\tilde{C}({R_1})$ with $\alpha(\epsilon)=\pi(\beta(\epsilon))$ for all $0<\epsilon<\epsilon_0$. We have $\pi(\tilde{t}_0)=t_0$ for $\tilde{t}_0=\lim_{\epsilon\to0}\beta(\epsilon)\in\tilde{C}({R_1})$. Now we consider the base change $\tilde{X}\to\tilde{C}$ of $X\to T$ which is a proper flat family over the smooth curve $\tilde{C}$. For all $\epsilon\in(0,\epsilon_0)$, the fiber of $\tilde{X}\to\tilde{C}$ over $\beta(\epsilon)$ is $X_{\alpha(\epsilon)}$ which satisfies $\N(X_{\alpha(\epsilon)})\leq m$. Therefore, the ${R_2}$-point of $\tilde{C}$ corresponding to the germ of $\beta$ corresponds to a smooth curve $X_{R_2}$ over ${R_2}$ with $\N(X_{R_2})\leq m$ by \cite{basuAlgorithmsRealAlgebraic2006}*{Proposition 3.17} and because $\N(T,m)$ is semialgebraic. On the other hand, since $\tilde{C}$ is smooth of dimension $1$, its local ring at $\tilde{t}_0\in\tilde{C}({R_1})$ can be naturally embedded into $A={R_1}\llbracket t \rrbracket^{\mathrm{alg}}$. Base change of $\tilde{X}\to\tilde{C}$ to $A$ gives a flat family over $A$ whose base change to ${R_2}$ is $X_{R_2}$ and whose special fiber is $X_{t_0}$. Now part $(ii)$ of \Cref{thm:criterioninf} implies $\N(X_{t_0})\leq m$ and thus $t_0\in \N(T,m)$.
\end{proof}

We can now conclude the section with the proof of \Cref{thm:B}.

\begin{proof}[Proof of \Cref{thm:B}]
    {We start by proving (ii).}
    Let $R_1=\RR$ and $R_2$ be the field of algebraic Puiseux series over $\RR$.
    Consider an admissible topological type $(g,r,a)$ with $0<r<g$ and, as in the proof of \Cref{thm:A_complete}, choose a smooth curve $X_{R_2}$ over $R_2$ of this topological type with $\N(X_{R_2})=\infty$. Next, we embed $X_{R_2}$ to $\PP^n_{R_2}$ via any complete linear system of degree $d\geq 2g+1$ and $n=d-g$. The corresponding $R_2$-point of the Hilbert scheme can be realized as the germ of a   semialgebraic path $h\colon (0,1]\to \cH_{(g,r,a)}$. For $\epsilon\in(0,1]$ we denote the smooth curve over $\RR$ of topological type $(g,r,a)$ that corresponds to the point $h(\epsilon)$ by $X_\epsilon$. By \Cref{lem:unboundediffinf} we have that $\N(X_\epsilon)$ is unbounded for $\epsilon\to0$, {concluding the proof of (ii)}.

    {We now prove (i). Let
    \[
        \N(g, m) = \left\{ \, X \in \scrM_g^{\RR} \mid \N(X) \le m  \, \right\}
    \]
    If $g=1$, it follows from the results of Huisman \cite{huismanGeometryAlgebraicCurves2001} and Monnier \cite{monnierDivisorsRealCurves2003} (see also \Cref{thm:huismanmonnier}) and by $m \ge 2g+1$, that $\N(1, m) = \scrM_{(1,1,1)}^{\RR} \sqcup \scrM_{(1,2,0)}^{\RR}$, which is closed. Assume then that $g \ge 2$.} {It is proven in \cite{degaayfortmanRealModuliSpaces2022a}*{Th.8.2} that the coarse moduli space $\scrM_g^\RR$ is homeomorphic to $\abs{\cM_g(\RR)}$, where $\cM_g$ is the moduli stack over $\RR$ of smooth genus $g$ curves. Here the topology on $\abs{\cM_g(\RR)}$ is the quotient topology of the Euclidean topology on $T(\RR)$ by the map $f_\RR \colon T(\RR) \to \abs{\cM_g(\RR)}$ induced by a smooth presentation $f \colon T \to \cM_g$ by an $\RR$-scheme $T$ for which $f_\RR \colon T(\RR) \to \abs{\cM_g(\RR)}$ is surjective. Such a scheme $T$ exists by \cite{degaayfortmanRealModuliSpaces2022a}*{Th.~7.3} and it is locally of finite type over $\RR$ because $\cM_g$ is of finite type. If we define $\N(T, m)$ as in \Cref{thm:semicontinuity}, then $\N(T, m) = f_\RR^{-1}(\N(g,m))$ and the set $\N(T, m)$ is closed by \Cref{thm:semicontinuity}. Thus $\N(g,m)$ is closed in $\abs{\cM_g(\RR)} \cong \scrM_g^\RR$.}
\end{proof}
\section{Quantitative analysis: topology and geometry}
\label{sec:quantitiative}
While in \Cref{sec:qualitative} we studied the boundedness or unboundedness of the totally real divisor threshold, we now focus on concrete bounds in the case of curves over $\RR$.
\subsection{Curves with many components}
\label{sec:many_ovals}
In this section, we consider curves with many connected components, i.e., smooth real curves $X$ of genus $g$ such that $r(X) = g$ or $r(X) = g+1$. Our goal is to reprove results of Huisman and Monnier, showing that the totally real divisor threshold is bounded from above by $2g-1$ in these cases.

Recall that the identity component $J(\RR)_0$ of the real locus of the Jacobian $J = \jac(X) \cong \Pic^0(X)$ consists of those classes $[D]$ of conjugation-invariant divisors $D$ of degree $0$ whose restriction to each connected component of $X(\RR)$ is of even degree, see e.g.~\cite{grossRealAlgebraicCurves1981a}. We will need the following, more refined characterization of $J(\RR)_0$.

\begin{proposition}\label{thm:generationIdComp}
    Let $X$ be a smooth real curve and $J = \jac(X)$. Suppose that $X(\RR)$ has $g$ or $g+1$ connected components, and let $S_1,\dots,S_g$ be $g$ of these components. Fix points $P_i\in S_i$ for $i=1,\dots,g$, then 
    \[
        J(\RR)_0=\biggl\{\sum_{i=1}^g [Q_i-P_i]\ |\ Q_i\in S_i\text{ for }i=1,\dots,g\biggr\}.
    \]
\end{proposition}

\begin{proof}
    Let $T=\jac(\RR)_0\cong(\RR/\Z)^g$, a $g$-dimensional real torus. Since $S_1,\ldots,S_g$ do not disconnect $X$, their fundamental classes can be extended to an integral homology basis of $H_1(X,\Z)$ by \Cref{lem:independent_real_components}. Let $S=S_1\times\cdots\times S_g \cong (\mathbb{S}^1)^g$, another $g$-dimensional real torus. Fixing base points $P_1,\dots,P_g$ as above, we obtain a map 
    \begin{align*}
        \varphi\colon  S & \longrightarrow T \\
        (Q_1,\dots,Q_g)& \longmapsto \sum_{i=1}^g [Q_i-P_i].
    \end{align*}
    Since $S_1,\dots,S_g$ are part of a basis of $H_1(X,\Z)$, the induced map $\varphi_\ast\colon H_1(S,\Z)\to H_1(T,\Z)$ is an isomorphism. The cohomology rings of $S$ and $T$ are generated in degree $1$ (see \cite{hatcherAlgebraicTopology}*{3.11}), hence it follows from Poincaré duality that $\varphi_\ast\colon H_g(S,\Z)\to H_g(T,\Z)$ is an isomorphism, as well. This implies that $\varphi$ is a map of topological degree $1$. In particular, it is surjective, concluding the proof.
\end{proof}

\Cref{thm:generationIdComp} allows us to give a unified proof of results due to Huisman \cite{huismanGeometryAlgebraicCurves2001} and Monnier \cite{monnierDivisorsRealCurves2003}.

\begin{theorem}[{\cites{huismanGeometryAlgebraicCurves2001,monnierDivisorsRealCurves2003}}]
    \label{thm:huismanmonnier}
        Let $X$ be a smooth real curve of genus $g$ such that $X(\RR)$ has $r=g$ or $r=g+1$ connected components. Then $\N(X) \le 2g-1$.
\end{theorem}

\begin{proof}
    Let $D$ be a conjugation-invariant divisor of degree $d\ge 2g-1$ on $X$. Let $S_1,\dots,S_r$ be the connected components of $X(\RR)$ and fix points $P_i\in S_i$ for each $i$ as above. Let $m$ be the number of components for which the restriction of $D$ to $S_i$ has odd degree. We may relabel and assume that these are the components $S_1,\dots,S_m$. Consider the divisor $E=D - (P_1+\cdots+P_m) - 2(P_{m+1}+\dots+P_g)$. By our choices, it has even degree on each connected component and hence its total degree $d-m-2(g-m)=d-2g+m\ge -1$ is also even, say $2k\ge 0$. It follows that the class of $E-2kP$ is contained in $J(\RR)_0$ for any point $P\in X(\RR)$. By \Cref{thm:generationIdComp}, there are points $Q_1,\dots,Q_g$ with $Q_i\in S_i$ such that $E-2kP\equiv \sum_{i=1}^g (Q_i-P_i)$, so that
    \[
     D\equiv 2kP+P_{m+1}+\cdots+P_g + Q_1+\cdots+Q_g.\qedhere
    \]
\end{proof}

\subsection{Bounds from metric properties}
\label{sec:lower_bound}
Let $X$ be a smooth real curve. We denote by $\len(X(\RR))$ the length of $X(\RR)$ with respect to the Bergman metric and by $\vol(J(\RR))$ the volume of the real part of the Jacobian $J = \jac(X)$ with respect to the canonical metric (see \Cref{sec:Bergman} for the definitions).
As a warm-up, we prove a lower bound on $\N(X)$ when $X(\RR)$ is connected. Recall that we always assume that $X(\RR) \neq \emptyset$ and that the Abel--Jacobi map $\phi$ is defined with base point $P_0 \in X(\RR)$, so that $\phi(X(\RR)) \subset \jac(X)(\RR)$.

\begin{proposition}
    \label{thm:lower_bound_r=1}
    Let $X$ be a smooth real curve such that $X(\RR)$ is connected and $J = \jac(X)$ its Jacobian. Then
    \[
        \N(X) \ge \frac{\vol(J(\RR))^{1/g}}{\len(X(\RR))}
    \]
\end{proposition}
\begin{proof}
    Recall that $\len(X(\RR))$ is equal, by definition, to the length of $\phi(X(\RR)) \subset J(\RR)$, the image of $X(\RR)$ under the Abel--Jacobi map. By \Cref{cor:comessatti}, since $X(\RR)$ has a single connected component, we have $J(\RR) = J(\RR)_0 \cong \RR^g / \Lambda$. The curve $\phi(X(\RR))$ is contained in a domain of the form $B/ \Lambda \subset J(\RR)$, where $B$ is a $g$-dimensional hypercube with side length equal to $\len(X(\RR))$. Therefore, the Minkowski-type sum $m \phi(X(\RR)) = \sum_{i=1}^m \phi(X(\RR))$ is contained in $m B / \Lambda$, and we have:
    \[
        \vol(m \phi(X(\RR))) \le \vol(m B \big/ \Lambda) \le \vol(m B) = (m\len(X(\RR)))^g
    \]
    where $\vol(m B)$ is computed with respect to the (flat) Euclidean metric of $\RR^g\cong H^0(X, \Omega)(\RR)$ induced by the Hodge inner product.

    Now, if $\vol(m \phi(X(\RR)))$ is smaller than $\vol(\jac(\RR))$, then $m < \N(X)$ (see \Cref{lem:0}). From the above inequality, this is implied by $(m\len(X(\RR)))^g < \vol J(\RR)$.  The contrapositive of this implication
    is $\N(X)\le m \Rightarrow \frac{\vol(\jac(\RR))^{1/g}}{\len(X(\RR))}\le m$. Letting $m=\N(X)$ shows the claim.
\end{proof}

We now generalize \Cref{thm:lower_bound_r=1} to real curves with an arbitrary number of connected components, proving \Cref{thm:C}.
\begin{proof}[{Proof of \Cref{thm:C}}]
    Let $X(\RR) = S_1 \sqcup \dots \sqcup S_r$ be the connected components of the real locus, and let $Z_i = \phi(S_i)$. Notice that 
    \[
        m \phi(X(\RR)) = \sum_{\substack{\alpha \in \NN^r \\ \abs{\alpha}=m}} \sum_{i=1}^r \alpha_i Z_i 
    \]
    where the addition is in the group law of $J = \jac(X) \cong \CC^g / \Lambda$. Recall that a divisor class belongs to the identity component $\jac(\RR)_0$ if and only if it has even degree on each connected component $S_i$. Therefore, we obtain:
    \begin{equation}\label{eq:sum_identity}
        2m \phi(X(\RR)) \cap \jac(\RR)_0 = \sum_{\substack{\alpha \in (2\NN)^r \\ \abs{\alpha} = 2m}}\sum_{i=1}^r \alpha_i Z_i = 2 \sum_{\substack{\alpha \in \NN^r \\ \abs{\alpha}=m}} \sum_{i=1}^r \alpha_i Z_i.
    \end{equation}
    
    Notice that, if $\ell_i = \len(Z_i) = \len(S_i)$, then there exists a $g$-dimensional hypercube $B_i$ with side length $\ell_i$ such that $Z_i \subset B_i / \Lambda$. Therefore $\sum_{i=1}^r Z_i \subset B / \Lambda$ for some $g$-dimensional hypercube $B$ with side length equal to $\len(X(\RR)) = \sum_{i=1}^r \ell_i$. We deduce
    \begin{equation}\label{eq:volume_cube}
        \vol(k \sum_{i=1}^r Z_i) \le \vol(k B \big/ \Lambda) \le (k\len(X(\RR)))^g
    \end{equation}
    for all $k \in \NN$. Hence we obtain
    \begin{align*}
        \vol\left( 2m \phi(X(\RR)) \cap \jac(\RR)_0 \right) &= \vol \left(2 \sum_{\substack{\alpha \in \NN^r \\ \abs{\alpha}=m}} \sum_{i=1}^r \alpha_i Z_i \right)  & \text{using \eqref{eq:sum_identity}}\\
        & = 2^g \vol \left( \sum_{i=1}^r \left(\sum_{\substack{\alpha \in \NN^r \\ \abs{\alpha}=m}}\alpha_i\right) Z_i \right) & \\
        & = 2^g \vol \left( \frac{m}{r}\binom{m+r-1}{r-1}\sum_{i=1}^r Z_i \right) &\\
        & \le 2^g \left(\frac{m}{r} \binom{m+r-1}{r-1} \len(X(\RR))\right)^g & \text{using \eqref{eq:volume_cube}}\\
        & \le 2^g \left(\frac{m}{r}\len(X(\RR))\right)^g \left(\frac{\e (m+r-1)}{r-1}  \right)^{g(r-1)} & \text{since $\binom{n}{k}\le (\e n /k)^k$} \\
        & \le \left(\frac{2\len(X(\RR))}{r}\right)^g \left(\frac{\e}{r-1}\right)^{g(r-1)}(m+r-1)^{gr}
    \end{align*}
    (here, we use the convention $(\e n/k)^k=1$ for $k=0$).
    Now, if
    \begin{equation*}
        \left(\frac{2\len(X(\RR))}{r}\right)^g \left(\frac{\e}{r-1}\right)^{g(r-1)}(m+r-1)^{gr} < \vol(\jac(\RR)_0),
    \end{equation*}
    then $2m \phi(X(\RR)) \cap \jac(\RR)_0 \subsetneq \jac(\RR)_0$, and thus $\N(X) > 2m$ (see \Cref{lem:0}).
    The contrapositive of this implication
    is 
    \begin{equation*}
    \frac{\N(X)}{2}\leq m \Rightarrow \left(\frac{\vol(\jac(\RR)_0)}{\len(X(\RR))^g} \left(\frac{r}{2}\right)^g \left(\frac{r-1}{\e}\right)^{g(r-1)}\right)^{\frac{1}{gr}}+1-r \le m.
    \end{equation*}
    Letting $m=\ceil{{\N(X)}/{2}}$, this implies the claim.
\end{proof}

\section{A genus two example}
\label{sec:genus2}
We now discuss in detail an example in genus $2$. For $0<\epsilon<1$, let \[f_\epsilon(t) = (1+t)((1-\epsilon)^2+t)((1+\epsilon)^2+t)\] and consider the family of curves $X_\epsilon$ with hyperelliptic model \[w^2 = f_{\epsilon}(z^2)\] For more details on real hyperelliptic curves, we refer the reader e.g.~to \cite{grossRealAlgebraicCurves1981a}*{Sec.~6}. We will show that the family $X_\epsilon$ is a quantitative, explicit version of \cite{benoistIntegralHodgeConjecture2020}*{Rem.~9.26}.

We first note that for every $0<\epsilon<1$ the polynomial $f_\epsilon(z^2)$ does not have multiple roots and is strictly positive on $\RR$. This implies that $X_\epsilon$ is a smooth curve of topological type $(2,1,0)$. We can realize them as in \Cref{sec:comessatti} by taking $\tilde{g} = 1$, $r=1$ and $m=0$. A period matrix using a Comessatti basis has the structure
\begin{equation*}
    \Pi^{(\epsilon)} = \left( \begin{array}{cc|cc}
        1 & {0} & {iT^{(\epsilon)}_{11}} & {\frac{1}{2}+iT^{(\epsilon)}_{12}} \\
        {0} & {1} & {\frac{1}{2}+iT^{(\epsilon)}_{12}} & iT^{(\epsilon)}_{22}
    \end{array} \right)
\end{equation*}
where the matrix $T^{(\epsilon)}=(T^{(\epsilon)}_{ij})_{i,j}$ is real symmetric and positive definite for all $0<\epsilon<1$.

Next, we will determine the stable curves $X_0$ and $X_1$ in the limit for $\epsilon\to0^+$ and $\epsilon\to1^-$, respectively. The curve defined by $w^2=f_1(z^2)$ has one singularity at $(0,0)$ and this is an ordinary node. Its normalization is a smooth curve of topological type $(1,2,0)$. This shows that the stable curve $X_1$ is the curve defined by $w^2=f_1(z^2)$. 

It is more involved to determine the curve $X_0$ because the polynomial $f_0(z^2)$ has two roots of multiplicity $3$, so that $w^2=f_0(z^2)$ does not define a stable curve. For $0<\epsilon<1$ the canonical map $X_\epsilon\to\PP^1$ has the six branch points given by the columns of the matrix
\begin{equation*}
A_\epsilon= \begin{pmatrix}
  1&1&1&1&1&1 \\
  i&-i&i(1-\epsilon)&-i(1-\epsilon)&i(1+\epsilon)&-i(1+\epsilon)
 \end{pmatrix}.
\end{equation*}
This defines a family of rational curves with six marked points and we first compute the limit $Y$ for $\epsilon\to0^+$ of this family in $\overline{\mathcal{M}}_{0,6}$ using the connection to regular matroid subdivisions developed in \cite{chowquotientsI}*{Sec.~4.1}. Treating $\epsilon$ as a formal variable with valuation $1$, the matrix $A_\epsilon$ defines a height function on the vertices of the hypersimplex
\begin{equation*}
 \Delta(2,6)=\conv(e_i+e_j\mid 1\leq i<j\leq 6)\subset\RR^6
\end{equation*}
which assigns to $e_i+e_j$ the valuation of the determinant of the submatrix of $A_\epsilon$ with columns indexed by $i$ and $j$. Using the software \texttt{polymake} \cite{polymake} we compute the induced regular matroid subdivision of $\Delta(2,6)$. It consists of two maximal cells that are interchanged by the permutation of the coordinates given by
$\sigma=(12)(34)(56)$. This shows that the complexification of the limit curve $Y$ in $\overline{\mathcal{M}}_{0,6}$ consists of two irreducible components intersecting in one point and each with three marked smooth points. Complex conjugation acts on the columns of $A_\epsilon$ as the permutation $\sigma$. Thus the two irreducible components of $Y_\C$ are interchanged by the complex conjugation.
By \cite{fultonNotesStableMaps1997}*{Prop.~6} (``properness of the Deligne--Mumford stack of stable maps''), the limit of the stable maps $X_\epsilon\to\PP^1$, where we mark on $X_\epsilon$ the ramification points, exists as a stable map $Z\to Y$ that is branched along the special points of $Y$ and of degree $2$. This shows that the complexification of $Z$ consists of two smooth curves of genus $1$, each with three smooth marked points, that are exchanged by the complex conjugation and which meet transversally at one point, namely the point that is mapped to the singularity of $Y$. Since $Z$ remains stable when forgetting about the marked points, it follows that $Z$ is equal to $X_0$.
This shows that the family $X_\epsilon$ is actually a concrete realization of \cite{benoistIntegralHodgeConjecture2020}*{Rem.~9.26}, see also \Cref{ex:div} for $r=k=1$. It follows that $\N(X_\epsilon)$ is unbounded for $\epsilon\to0^+$. In the next section, we will use the metric approach from \Cref{sec:lower_bound} to make this quantitative. We will also see that $\N(X_\epsilon)$ remains bounded for $\epsilon\to1^-$.

\subsection{A lower bound for \texorpdfstring{$\N(X_\epsilon)$}{}}
\label{sec:example_1} In the following, we want to apply \Cref{thm:lower_bound_r=1} to compute a lower bound for $\N(X_\epsilon)$. To do that, we need to explicitly obtain the period matrix $\Pi^{(\epsilon)}$. Details on finding the homology basis for hyperelliptic curves can be found e.g.~in \cite{farkasRiemannSurfaces1992}. A more general framework for period computations for real curves will become available in \cite{baldiComputingPeriodsReal}.
This problem for real hyperelliptic curves with $r=g+1$ connected components has also been studied in \cite{buserGeodesicsPeriodsEquations2001c}.

The branch points of the two-to-one map $X_{\epsilon} \to \PP^1$ are the roots of $f_{\epsilon}(z^2)$, i.e., $\pm i(1+\epsilon)$, $\pm i$ and $\pm i(1-\epsilon)$. A basis of $H^0(X, \Omega)$ invariant under complex conjugation is e.g.~$\eta_1 = \dd z/ w$ and $\eta_2 = z\dd z /w$. Following the notations of \Cref{sec:comessatti}, and using the substitution $z^2 = - t$, we can show that:

\begin{align*}
    \int_{b_1} \eta_1 &= \int_{\beta} \frac{\dd z}{w}=2i \int_{1}^{(1+\epsilon)^2} \frac{\dd t}{\sqrt{tf_\epsilon(-t)}} = \frac{2i}{\sqrt{\epsilon}}K\left(\frac{(-1+\epsilon)\sqrt{2+\epsilon}}{2}\right)\\
    \int_{b_2} \eta_1 &= \int_{\alpha} \frac{\dd z}{w} = 2\int_{(1-\epsilon)^2}^1 \frac{\dd t}{\sqrt{-tf_\epsilon(-t)}} = \frac{2}{\sqrt{\epsilon}}K\left(\frac{(-1+\epsilon)\sqrt{2+\epsilon}}{2}\right)\\
    \int_{a_1} \eta_1 &= \int_{\alpha+\alpha^\sigma} \frac{\dd z}{w} = \frac{4}{\sqrt{\epsilon}}K\left(\frac{(-1+\epsilon)\sqrt{2+\epsilon}}{2}\right)\\
    \int_{a_2} \eta_1 &= \int_{\beta+\beta^\sigma} \frac{\dd z}{w} = 0 \\
    \int_{b_1} \eta_2 &= \int_{\beta} \frac{z\dd z}{w} = 2 \int_{1}^{(1+\epsilon)^2} \frac{\dd t}{\sqrt{f_\epsilon(-t)}} = \frac{2}{\sqrt{\epsilon}}K\left(\frac{\sqrt{2+\epsilon}}{2}\right)\\
    \int_{b_2} \eta_2 &= \int_{\alpha} \frac{z\dd z}{w} =2i \int_{(1-\epsilon)^2}^{1} \frac{\dd t}{\sqrt{-f_\epsilon(-t)}} = \frac{2i}{\sqrt{\epsilon}}K\left(\frac{\sqrt{2-\epsilon}}{2}\right)\\
    \int_{a_1} \eta_2 &= \int_{\alpha+\alpha^\sigma} \frac{z\dd z}{w} = 0\\
    \int_{a_2} \eta_2 &= \int_{\beta+\beta^\sigma} \frac{z\dd z}{w} = \frac{4}{\sqrt{\epsilon}}K\left(\frac{\sqrt{2+\epsilon}}{2}\right)
\end{align*}
where $K(k) = \int_0^1 \frac{\dd t}{\sqrt{(1-t^2)(1-k^2t^2)}}$ is the complete elliptic integral of the first kind. Above and in the following, the computational results were obtained with the help of \texttt{Maple} \cite{maple} and \texttt{Mathematica} \cite{Mathematica}.
Normalizing the basis of holomorphic $1$-forms as $\omega_i = \eta_i / \left(\int_{a_i}\eta_i\right)$, we obtain the period matrix, with respect to $\omega_1, \omega_2$ and $a_1, a_2, b_1, b_2$:
\begin{equation*}
    \Pi^{(\epsilon)} = \left( \begin{array}{cc|cc}
        1 & {0} & { i \frac{K\left(\frac{(-1+\epsilon)\sqrt{2+\epsilon}}{2}\right)}{2K\left(\frac{(1+\epsilon)\sqrt{2-\epsilon}}{2}\right)}} & {\frac{1}{2}} \\
        {0} & {1} & {\frac{1}{2}} & i \frac{K\left(\frac{\sqrt{2-\epsilon}}{2}\right)}{2K\left(\frac{\sqrt{2+\epsilon}}{2}\right)}
    \end{array} \right)
\end{equation*}
An orthonormal basis of differentials, with respect to the canonical metric of the Jacobian (see \Cref{sec:Bergman}), is therefore:
\begin{align*}
    \theta_1^{(\epsilon)} &= \frac{\sqrt{\epsilon}}{4}\left({{2K\left(\frac{(-1+\epsilon)\sqrt{2+\epsilon}}{2}\right)K\left(\frac{(1+\epsilon)\sqrt{2-\epsilon}}{2}\right)}}\right)^{-\frac{1}{2}} \frac{\dd z}{w} \\ \theta_2^{(\epsilon)} &= \frac{\sqrt{\epsilon}}{4} \left({{2K\left(\frac{\sqrt{2-\epsilon}}{2}\right)K\left(\frac{\sqrt{2+\epsilon}}{2}\right)}}\right)^{-\frac{1}{2}} \frac{z\dd z}{w}
\end{align*}
The image of the Abel--Jacobi map with respect to $\theta_1^{(\epsilon)}$ and $\theta_2^{(\epsilon)}$, and the fundamental domain of the Jacobian are shown in \Cref{fig:example210}: since $\phi(X(\RR))$ is getting smaller and smaller, intuitively, $\N(X_\epsilon)$ is growing to infinity.
\begin{figure}
    \begin{center}
        \begin{tabular}{c c}
                \includegraphics[width=0.45\textwidth]{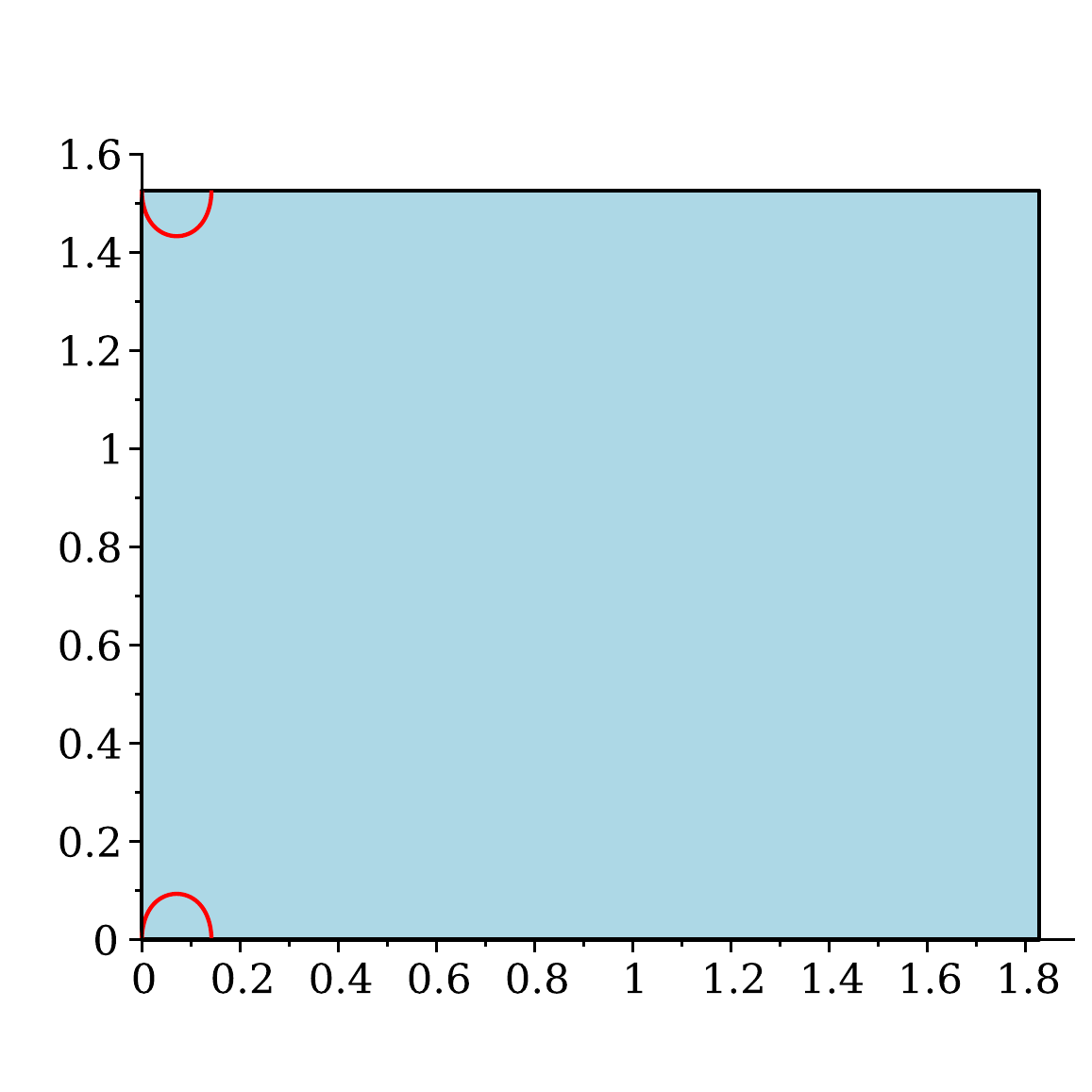} & \includegraphics[width=0.45\textwidth]{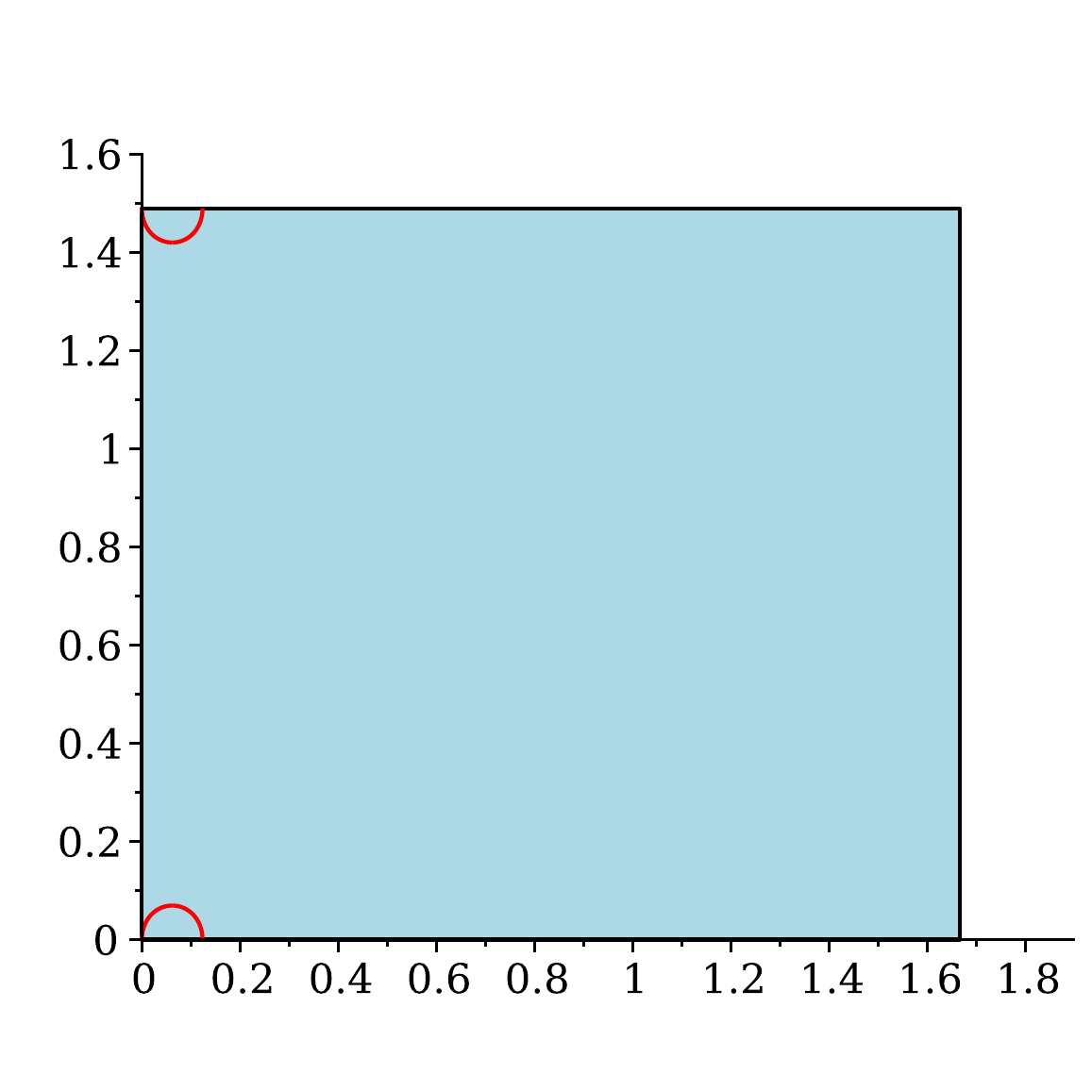}
                \\
                {$\epsilon=0.65$} & { $\epsilon=0.45$}
                \\  \includegraphics[width=0.45\textwidth]{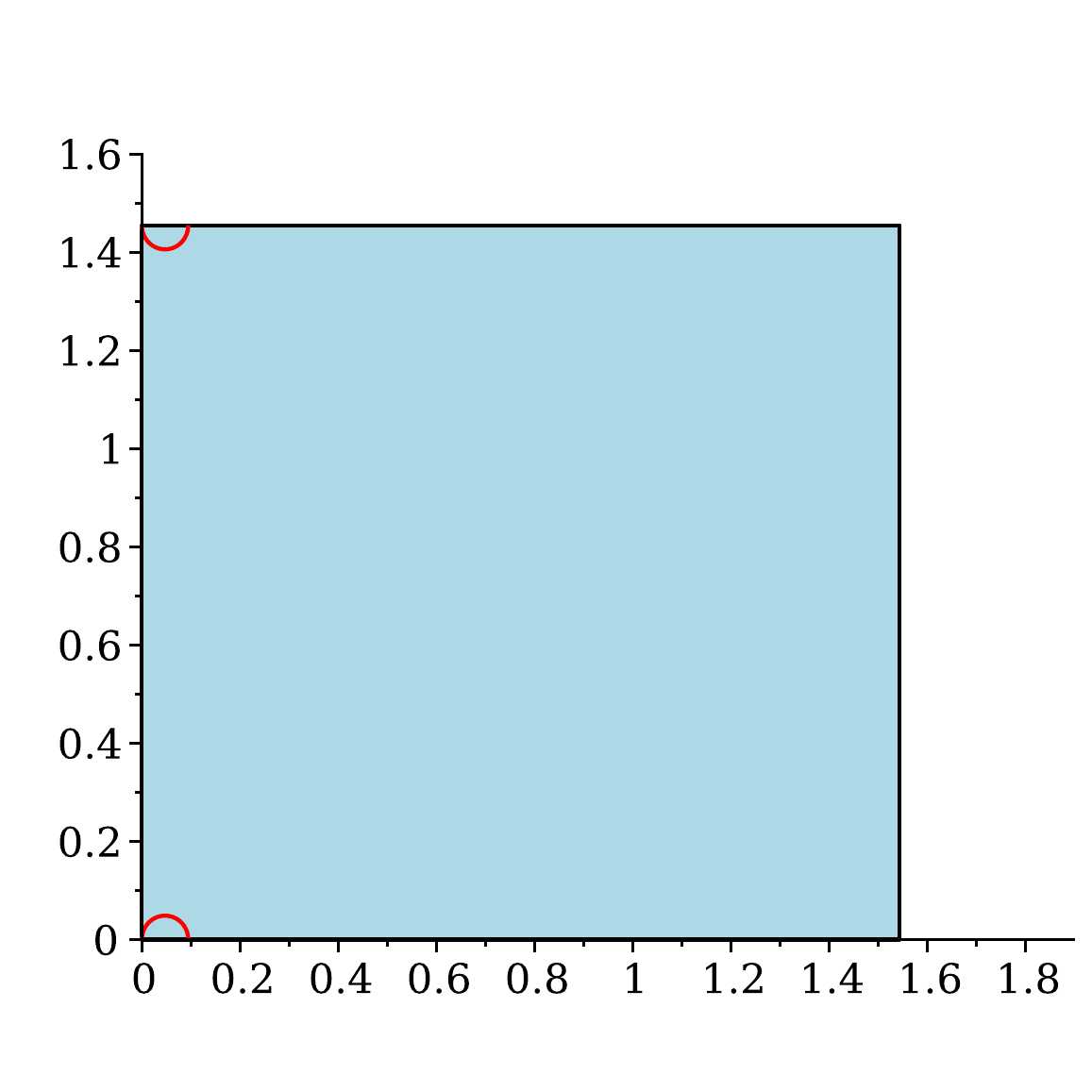} & \includegraphics[width=0.45\textwidth]{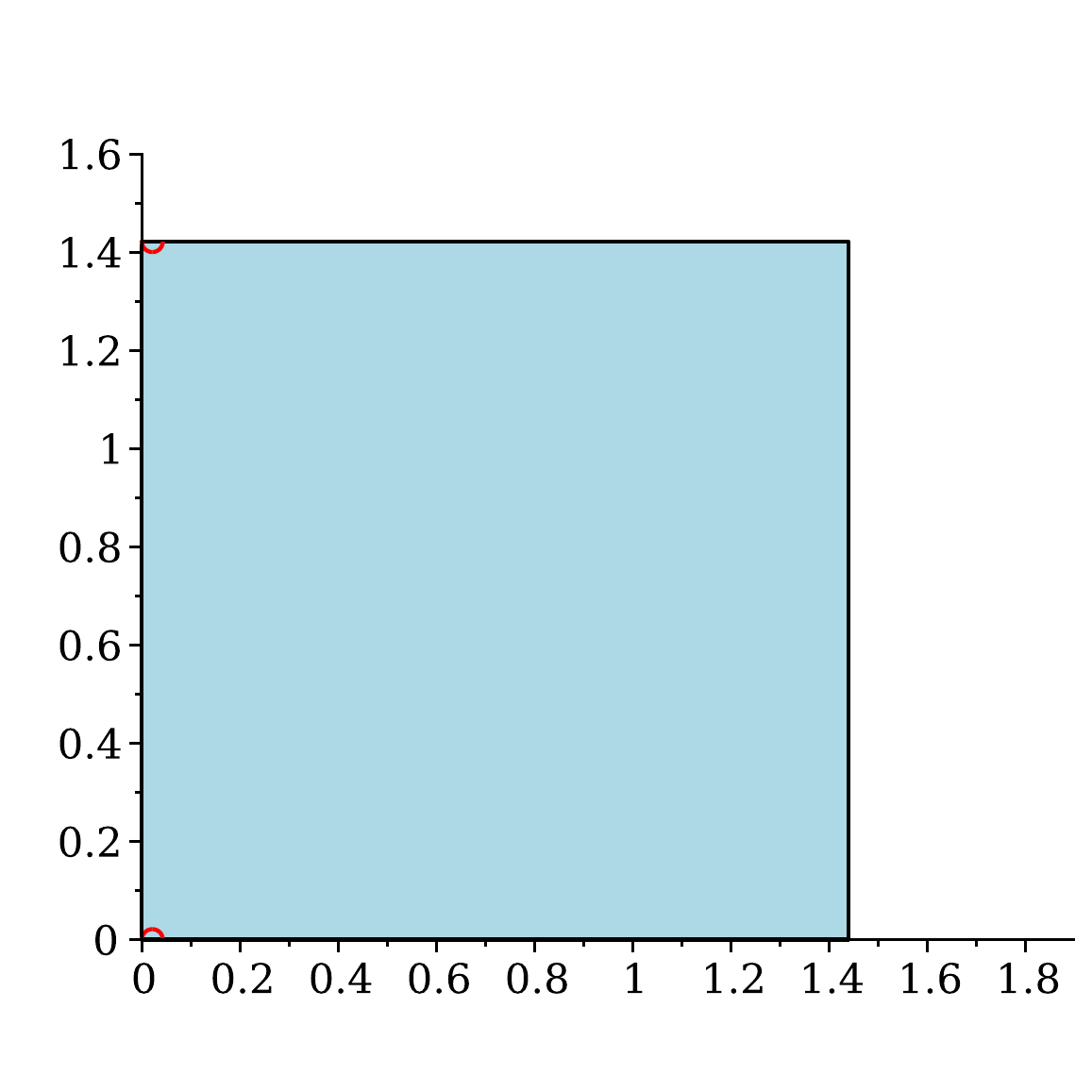}
                \\
                { $\epsilon=0.25$} & { $\epsilon=0.05$}
    \end{tabular}
\end{center}
\caption{For different values of $\epsilon$: in red, the image of $X_\epsilon(\RR)$ by the Abel--Jacobi map w.r.t. $\theta_1^{(\epsilon)}, \theta_2^{(\epsilon)}$; in blue, the fundamental domain of $J_{\epsilon}(\RR)$.
}
\label{fig:example210}
\end{figure}

We want to make the previous statement formal and quantitative using \Cref{thm:lower_bound_r=1}. The volume of the real part of the Jacobian $J_\epsilon = \jac(X_\epsilon)$ (which has only one connected componend by \Cref{cor:comessatti}) is, by \Cref{cor:volume_jacobian}:
\[
    \vol(J_\epsilon(\RR)) = \vol(J_\epsilon(\RR)_0 ) = 2\sqrt{\frac{K\left(\frac{(1+\epsilon)\sqrt{2-\epsilon}}{2}\right)K\left(\frac{\sqrt{2+\epsilon}}{2}\right)}{K\left(\frac{(-1+\epsilon)\sqrt{2+\epsilon}}{2}\right)K\left(\frac{\sqrt{2-\epsilon}}{2}\right)}} 
\]
and $\vol(J_\epsilon(\RR)) \to 2$ from above as $\epsilon \to 0^+$.

We now study $\len(X_\epsilon(\RR)) = \len(\phi(X_{\epsilon}(\RR)))$, where $\phi$ is the Abel--Jacobi map. Recall that, writing $\theta_1^{(\epsilon)} = \frac{\sqrt{\epsilon}}{4}h_1(\epsilon)\frac{\dd z}{w}$ and $\theta_2^{(\epsilon)} = \frac{\sqrt{\epsilon}}{4}h_2(\epsilon)\frac{z\dd z}{w}$, the Bergman metric (see \Cref{sec:Bergman}) is
\[
    \theta_1^{(\epsilon)}\overline{\theta}_1^{(\epsilon)} + \theta_2^{(\epsilon)}\overline{\theta}_2^{(\epsilon)} = \frac{\epsilon}{16} \ \frac{h_1(\epsilon)^2+h_2(\epsilon)^2\abs{z}^2}{\abs{w}^2} \ \dd z \dd \overline{z}.
\]
Writing $z = x+iy$ and since $X(\RR) \to \PP^1(\RR)$ is a double cover (because $f_\epsilon$ is positive on $\RR$), we get
\begin{align*}
    \len(X_\epsilon(\RR)) & = 2 \int_{-\infty}^{{\infty}} \sqrt{\frac{\epsilon}{16} \ \frac{h_1(\epsilon)^2+h_2(\epsilon)^2x^2}{f_\epsilon(x^2)}} \dd x = \sqrt{\epsilon} \int_{0}^{{\infty}} \sqrt{\frac{h_1(\epsilon)^2+h_2(\epsilon)^2x^2}{f_\epsilon(x^2)}} \dd x \\ & = \frac{\sqrt{\epsilon}}{2} \int_{0}^{{\infty}} \sqrt{\frac{h_1(\epsilon)^2+h_2(\epsilon)^2t}{tf_\epsilon(t)}}\dd t \le \frac{\sqrt{\epsilon}h_1(0)}{2} \int_{0}^{{\infty}} \sqrt{\frac{1+t}{tf_\epsilon(t)}}\dd t \\
    & = \frac{\sqrt{2}}{4K\left(\frac{\sqrt{2}}{2}\right)}\sqrt{\epsilon} \ \underbrace{\left(\frac{2 i K \left(\frac{\epsilon -1}{1+\epsilon}\right) \sqrt{\epsilon}-4 \sqrt{\epsilon} K \left(\frac{2 \sqrt{\epsilon}}{1+\epsilon}\right)+K \left(\frac{1+\epsilon}{2 \sqrt{\epsilon}}\right) \epsilon +K \left(\frac{1+\epsilon}{2 \sqrt{\epsilon}}\right)}{\left(1+\epsilon \right) \sqrt{\epsilon}}\right)}_{\textstyle{\eqqcolon \ell(\epsilon) \xrightarrow[\epsilon \to 0^+]{}\pi}}
\end{align*}
where the middle inequality holds true since $h_j(\epsilon) \le h_1(0)$ for $j=1,2$ and all $0<\epsilon<1$.
Combining \Cref{thm:lower_bound_r=1} with the above computations, and since $\ell(\epsilon) <4$ for $0 < \epsilon < 1/2$, we finally conclude that
\begin{equation}
    \label{eq:numerical_lower_bound}
    \N(X_\epsilon) \ge \frac{K\left(\frac{\sqrt{2}}{2}\right)}{\sqrt{\epsilon}} = \frac{1.854\ldots}{\sqrt{\epsilon}} 
\end{equation}
for $0<\epsilon <1/2$, or, more precisely, $\N(X_\epsilon)$ grows at least as $4K\left(\frac{\sqrt{2}}{2}\right)(\pi \sqrt{\epsilon})^{-1}$ for $\epsilon \to 0^+$.
In the next section, we show that this lower bound is asymptotically tight, up to a constant.

\subsection{Refined analysis}
\label{sec:refined}
While in \Cref{sec:example_1} we applied general results to the hyperelliptic family $X_\epsilon$, in this section we develop a refined analysis for this special example. The discussion will also show possible future research directions.

We will need the following elementary lemma.
\begin{lemma}\label{lem:plane_convex}
    Let $\Gamma \subset \RR^2$ be a, smooth, closed, simple curve enclosing a convex region of $\RR^2$. Then for all $\alpha \in \conv(\Gamma)$ there exists $\gamma_1, \gamma_2 \in \Gamma$ such that $\alpha = \frac{1}{2}(\gamma_1+\gamma_2)$.
\end{lemma}
\begin{proof}
    This follows from the intermediate value theorem.
\end{proof}

The differential of the Abel--Jacobi map, in our case expressed using $\theta_1^{(\epsilon)}$ and $\theta_1^{(\epsilon)}$ and with base point $P_0^{(\epsilon)} = \lim_{t\to\infty} (-t,\sqrt{f_{\epsilon}(t^2)}) \in X_\epsilon(\RR)$, is the (affine) canonical map. After shifting the fundamental domain of the Jacobian as in \Cref{fig:example210bis}, it is then possible to show that $\phi(X_\epsilon(\RR)) \subset \RR^2$ satisfies the hypotheses of \Cref{lem:plane_convex}.
\begin{figure}
    \begin{center}
        \begin{tabular}{c c}
                \includegraphics[width=0.45\textwidth]{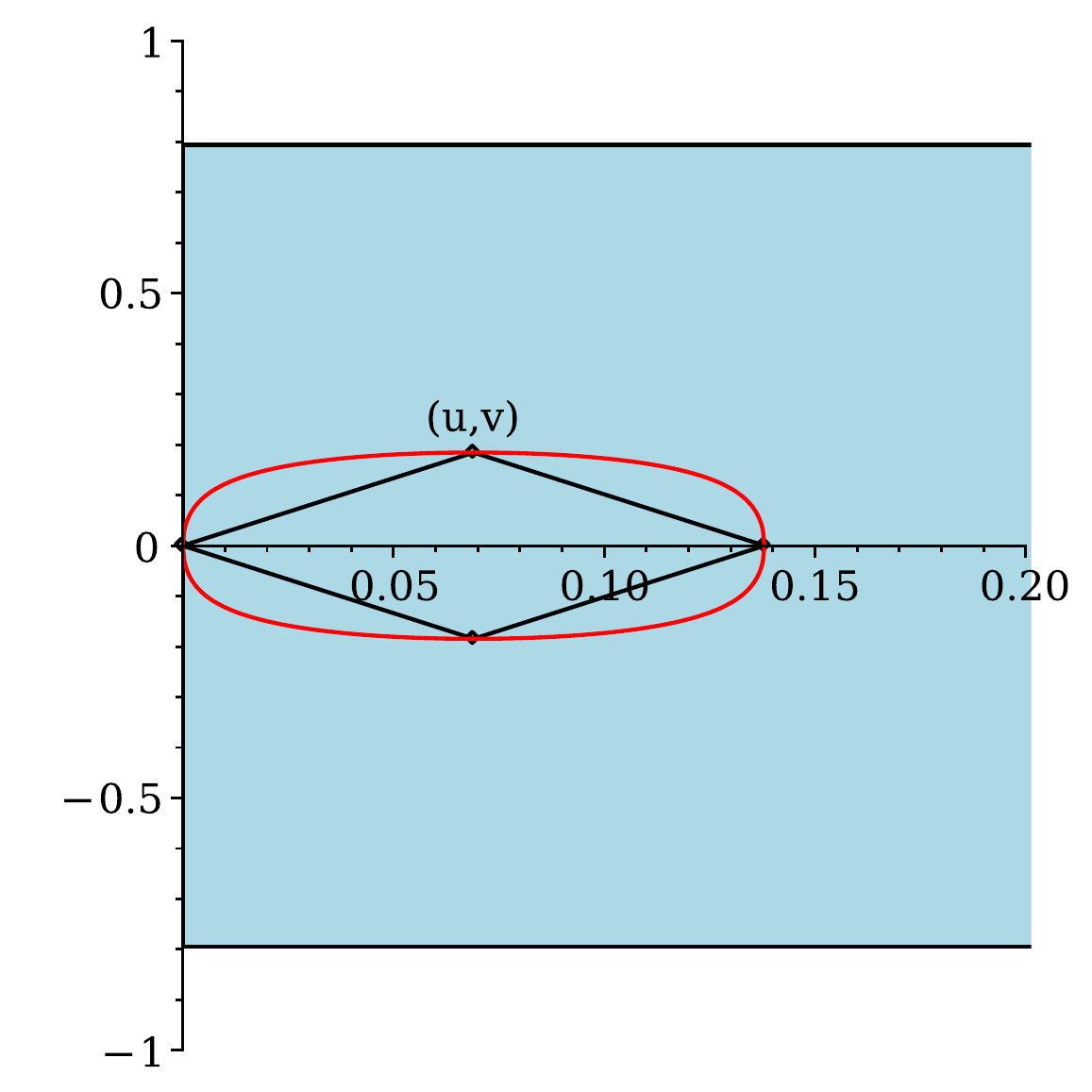} & \includegraphics[width=0.45\textwidth]{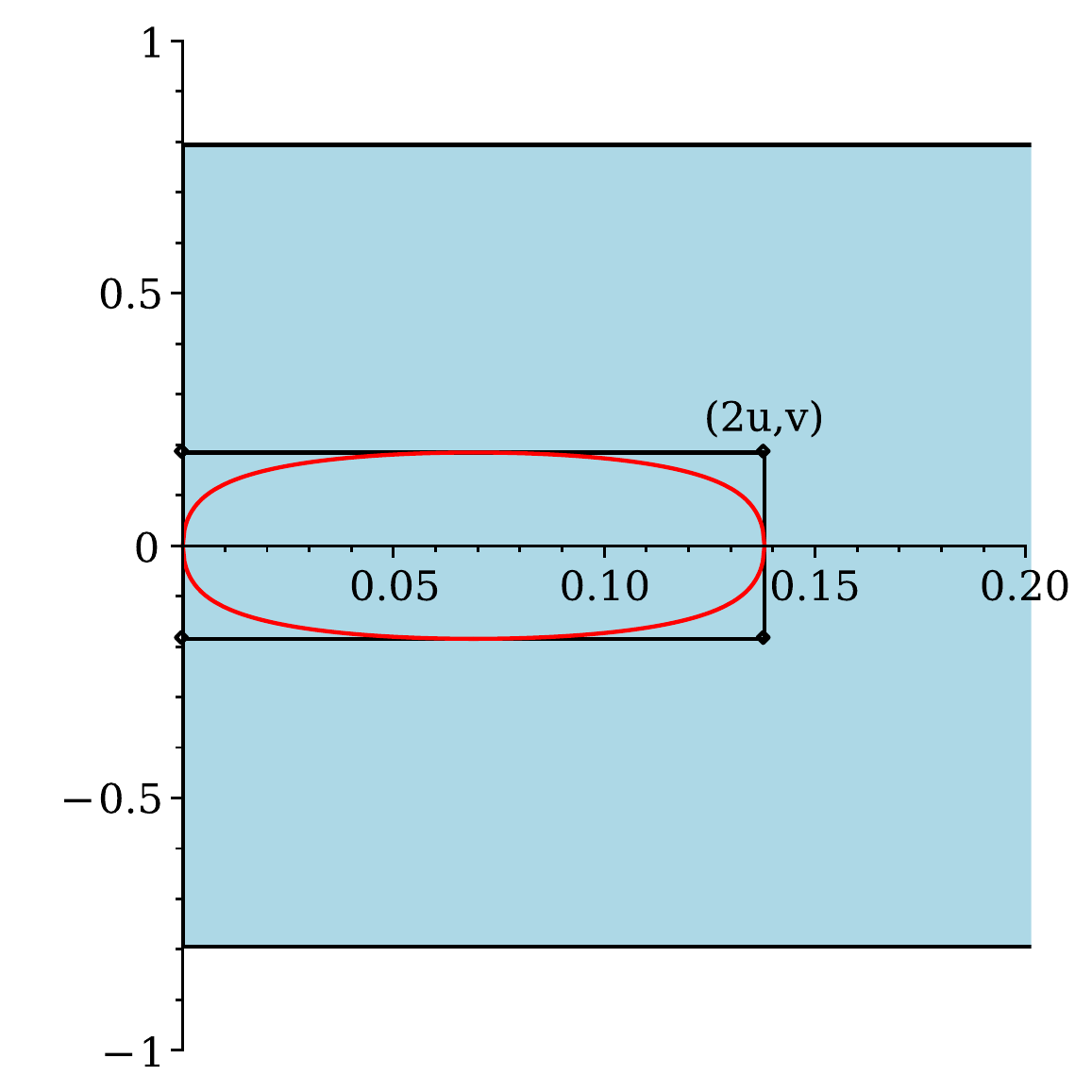}
\end{tabular}
\caption{For $\epsilon=0.95$: inner and outer approximation of $\conv(\phi(X_\epsilon(\RR)))$
}
\label{fig:example210bis}
\end{center}
\end{figure}

We now give a simple, polyhedral inner approximation of $\conv(\Gamma) \subset \RR^2$. Writing $\theta_1^{(\epsilon)} = \frac{\sqrt{\epsilon}}{4}h_1(\epsilon)\frac{\dd z}{w}$ and $\theta_2^{(\epsilon)} = \frac{\sqrt{\epsilon}}{4}h_2(\epsilon)\frac{z\dd z}{w}$ as before, let
\begin{align*}
    u=u(\epsilon) & = \frac{\sqrt{\epsilon}}{4}h_1(\epsilon) \int_{0}^{{\infty}}\frac{\dd x}{\sqrt{f_\epsilon(x^2)}} = \frac{\sqrt{\epsilon}}{8}h_1(\epsilon) \int_{0}^{{\infty}}\frac{\dd t}{\sqrt{tf_\epsilon(t)}} \\ & = \frac{h_1(\epsilon)}{4} \frac{F\left(\sqrt{2\epsilon-\epsilon^2}, \frac{2}{(\epsilon+1)\sqrt{2-\epsilon}}\right)}{(\epsilon+1)\sqrt{2 - \epsilon}} = \frac{1}{4\sqrt{2}K\left(\frac{\sqrt{2}}{2}\right)}\sqrt{\epsilon} + o(\sqrt{\epsilon}) \\
    v=v(\epsilon) & = \frac{\sqrt{\epsilon}}{4}h_2(\epsilon) \int_{0}^{{\infty}}\frac{x\dd x}{\sqrt{f_\epsilon(x^2)}}  = \frac{\sqrt{\epsilon}}{8}h_2(\epsilon) \int_{0}^{{\infty}}\frac{\dd t}{\sqrt{f_\epsilon(t)}} \\ & = \frac{h_2(\epsilon)}{8}\left(K\left(\frac{\sqrt{\epsilon+2}}{2}\right)+F\left(\epsilon-1,\frac{\sqrt{\epsilon+2}}{2} \right)\right) = \frac{1}{4\sqrt{2}K\left(\frac{\sqrt{2}}{2}\right)}\sqrt{\epsilon} + o(\sqrt{\epsilon})
\end{align*}
where $F(a,k) = \int_0^a \frac{\dd t}{\sqrt{(1-t^2)(1-k^2t^2)}}$ is the incomplete elliptic integral of the first kind. The point $(u(\epsilon), v(\epsilon)) \in \phi(X(\RR))$ is obtained by integrating over one fourth of the total length of $X(\RR)$ (see \Cref{fig:example210bis}, left).

Let $P=P^{(\epsilon)} = \conv((0,0), (u,v), (2u,0), (u,-v))$ be the rhombus in \Cref{fig:example210bis}. Notice that $P^{(\epsilon)} \subset \phi(X_\epsilon(\RR))$ by construction. \Cref{lem:plane_convex} implies that $2m\phi(X_\epsilon(\RR)) \supset 2m P^{(\epsilon)}$ for all $m$. Therefore, if $2m P^{(\epsilon)}/ \Lambda = J_\epsilon(\RR)$, then $\N(X_\epsilon) \le 2m$. Since \[2m P^{(\epsilon)} = \conv((0,0), 2m(u,v), 2m(2u,0), 2m(u,-v))\] and the fundamental domain has side lengths
\[
\ell_1 = \sqrt{\frac{2K\left(\frac{(1+\epsilon)\sqrt{2-\epsilon}}{2}\right)}{K\left(\frac{(-1+\epsilon)\sqrt{2+\epsilon}}{2}\right)}} \quad\quad \ell_2 = \sqrt{\frac{2K\left(\frac{\sqrt{2+\epsilon}}{2}\right)}{K\left(\frac{\sqrt{2-\epsilon}}{2}\right)}}
\]
 one can verify that $2m P^{(\epsilon)}/ \Lambda = J_\epsilon(\RR)$ if
\[
    2m \ge \frac{\ell_1}{u} + \frac{\ell_2}{v} = 16 K\left(\frac{\sqrt{2}}{2}\right) \frac{1}{\sqrt{\epsilon}} + o\left(\frac{1}{\sqrt{\epsilon}}\right)
\]
or, if we want an exact expression at the price of a less accurate asymptotic constant, we can show that $2m P^{(\epsilon)}/ \Lambda = J_\epsilon(\RR)$ if
$
    2m \ge 19 K\left(\sqrt{2}/{2}\right) /{\sqrt{\epsilon}}
$. Then finally, for all $0<\epsilon <1$
\begin{equation}
    \label{eq:numerical_upper_bound}
    \N(X_\epsilon) \le 19 K\left(\frac{\sqrt{2}}{2}\right) \frac{1}{\sqrt{\epsilon}}
\end{equation}
Note that this also shows that $\N(X_\epsilon)$ remains bounded for $\epsilon\to1^-$.

We can also perform a similar analysis using the outer approximation of $\conv(\phi(X_\epsilon(\RR)))$ with the rectangle in \Cref{fig:example210bis}, right. In this way, we obtain a lower bound for $\N(X_\epsilon)$ (instead of an upper bound) of the form:
\[
    \N(X_\epsilon) \ge 4 K\left(\frac{\sqrt{2}}{2}\right) \frac{1}{\sqrt{\epsilon}} + o\left(\frac{1}{\sqrt{\epsilon}}\right)
\]
Notice that this lower bound is, up to a constant, asymptotically the same as the one we obtained in \Cref{sec:example_1} using \Cref{thm:lower_bound_r=1}. Combining the inequalities \eqref{eq:numerical_lower_bound} and \eqref{eq:numerical_upper_bound}, we conclude that $\N(X_\epsilon)$ grows asymptotically as $1/\sqrt{\epsilon}$, or more precisely that
\[
    \frac{1.854\ldots}{\sqrt{\epsilon}}=K\left(\frac{\sqrt{2}}{2}\right) \frac{1}{\sqrt{\epsilon}} \le \N(X_\epsilon) \le 19 K\left(\frac{\sqrt{2}}{2}\right) \frac{1}{\sqrt{\epsilon}} =\frac{35.22\ldots}{\sqrt{\epsilon}}
\]
holds for all $0<\epsilon<1/2$. In particular, for the family $X_\epsilon$ the bound in \Cref{thm:lower_bound_r=1} is asymptotically tight, up to a constant.

\smallskip
We conclude with the following remark: In the context of nonnegative polynomials on curves, real $2$-torsion points of the Jacobian are of special interest \cites{scheidererSumsSquaresRegular2000,baldiNonnegativePolynomialsMoment2024b}. Let $Q_1^{(\epsilon)}$, $Q_2^{(\epsilon)}$ and $Q_3^{(\epsilon)}$ be the three complex conjugate pairs of points on $X_\epsilon$ such that $Q_1^{(\epsilon)}+Q_2^{(\epsilon)}+Q_3^{(\epsilon)}$ is the ramification divisor of the canonical map $X_\epsilon\to\PP^1$. Then the three non-trivial real $2$-torsion points of $J_\epsilon$ are represented by the three divisors $Q_i^{(\epsilon)}-Q_j^{(\epsilon)}$ for $1\leq i<j\leq 3$. Using the outer approximation of $\conv(\phi(X_\epsilon(\RR)))$ with the box on the right of \Cref{fig:example210bis}, we can give a lower bound on the smallest natural number $m$ such that $Q_i^{(\epsilon)}-Q_j^{(\epsilon)}+m\cdot P_0$ is linearly equivalent to a totally real effective divisor. Indeed, if $m$ times this box contains a $2$-torsion point, then we must have $\frac{\ell_1}{2}\leq m\cdot 2u$ or $\frac{\ell_2}{2}\leq m\cdot v$. This is equivalent to one of the following inequalities being satisfied:
\begin{align*}
 K\left(\frac{(1+\epsilon)\sqrt{2-\epsilon}}{2}\right)\leq& m\cdot\frac{\sqrt{\epsilon}}{4}\cdot\int_{0}^{{\infty}}\frac{\dd t}{\sqrt{tf_\epsilon(t)}}\\
 K\left(\frac{\sqrt{2+\epsilon}}{2}\right)\leq& m\cdot\frac{\sqrt{\epsilon}}{8}\cdot\int_{0}^{{\infty}}\frac{\dd t}{\sqrt{f_\epsilon(t)}}
\end{align*}
For $0<\epsilon<\frac{1}{2}$ the elliptic integrals on the left-hand sides can be bounded from below by $1$ and we have that $f_\epsilon(t)\geq(1+t)(t+\frac{3}{4})^2$ for all $t\geq0$. Then one of the following inequalities are satisfied:
\begin{align*}
 1\leq m\cdot\frac{\sqrt{\epsilon}}{4}\cdot\int_{0}^{{\infty}}\frac{\dd t}{(t+\frac{3}{4})\sqrt{t(1+t)}}&=m\cdot\sqrt{\epsilon}\cdot\frac{\pi }{3 \sqrt{3}}\leq m\cdot\sqrt{\epsilon}\\
 1\leq m\cdot\frac{\sqrt{\epsilon}}{8}\cdot\int_{0}^{{\infty}}\frac{\dd t}{(t+\frac{3}{4})\sqrt{(1+t)}}&=m\cdot\sqrt{\epsilon}\cdot\frac{\log (9)}{8}\leq m\cdot\sqrt{\epsilon}
\end{align*}
In conclusion, the smallest natural number $m$ such that $Q_i^{(\epsilon)}-Q_j^{(\epsilon)}+m\cdot P_0$ is linearly equivalent to a totally real effective divisor satisfies $m\geq\frac{1}{\sqrt{\epsilon}}$ for all $0<\epsilon<\frac{1}{2}$.

\end{document}